\documentclass[reqno,a4paper]{amsart}
\usepackage[T1]{fontenc}
\usepackage[utf8]{inputenc}
\usepackage[english]{babel}
\usepackage{microtype}

\usepackage{latexsym, amssymb, amsmath}
\usepackage{enumitem}
\usepackage{amsthm}
\usepackage{tikz-cd}
\usepackage{orcidlink}
\usepackage{doi}
\renewcommand{\doi}[1]{DOI:~\href{https://doi.org/#1}{#1}}
\usepackage{url}

\usepackage{verbatim}
\usepackage[all]{xy}
\usepackage{calrsfs}
\usepackage{mathabx}
\usepackage{xcolor}
\usepackage{mathtools}
\usepackage{thmtools}
\usepackage{blkarray}
\usepackage{booktabs}
\usepackage[textsize=scriptsize]{todonotes}
\usepackage{xurl}
\usepackage{nameref,hyperref,cleveref}

\theoremstyle{plain}
\newtheorem{theorem}[subsection]{Theorem}
\newtheorem{proposition}[subsection]{Proposition}
\newtheorem{corollary}[subsection]{Corollary}
\newtheorem{lemma}[subsection]{Lemma}

\theoremstyle{definition}
\newtheorem{definition}[subsection]{Definition}

\theoremstyle{remark}
\newtheorem{remark}[subsection]{Remark}
\newtheorem{example}[subsection]{Example}
\newtheorem{examples}[subsection]{Examples}

\crefname{examples}{Example}{Examples}
\Crefname{examples}{Example}{Examples}
\crefname{problem}{Open problem}{Open problems}
\Crefname{problem}{Open problem}{Open problems}

\numberwithin{equation}{section}

\newenvironment{tfae}
{
\begin{enumerate}}
{\end{enumerate}}

\newcommand{\cC}{\ensuremath{\mathcal{C}}}

\newcommand{\A}{\ensuremath{\mathcal{A}}}

\newcommand{\bF}{\mathbb{F}}         
\newcommand{\E}{\mathcal{E}}

\newcommand{\N}{\mathbb{N}}
\newcommand{\V}{\mathcal{V}}

\DeclareMathOperator{\bchar}{char}

\DeclareMathOperator{\Aut}{Aut}
\DeclareMathOperator{\Hol}{Hol}
\DeclareMathOperator{\op}{op}
\DeclareMathOperator{\Act}{Act}
\DeclareMathOperator{\DAct}{DAct}
\DeclareMathOperator{\Hom}{Hom}
\DeclareMathOperator{\Bim}{Bim}
\DeclareMathOperator{\SplExt}{SplExt}

\DeclareMathOperator{\End}{End}
\DeclareMathOperator{\Der}{Der}
\DeclareMathOperator{\ADer}{ADer}
\DeclareMathOperator{\Bider}{Bider}
\DeclareMathOperator{\Inn}{Inn}
\DeclareMathOperator{\M}{M}

\DeclareMathOperator{\USGA}{USGA}
\DeclareMathOperator{\Ann}{Ann}
\DeclareMathOperator{\bFM}{FM}

\DeclareMathOperator{\id}{id}

\newcommand{\Set}{\ensuremath{\mathsf{Set}}}
\newcommand{\Rng}{\ensuremath{\mathsf{Rng}}}
\newcommand{\Mag}{\ensuremath{\mathsf{Mag}}}
\newcommand{\Pt}{\ensuremath{\mathsf{Pt}}}
\newcommand{\Assoc}{\ensuremath{\mathsf{Assoc}}}

\newcommand{\Lie}{\ensuremath{\mathsf{Lie}}}
\newcommand{\Leib}{\ensuremath{\mathsf{Leib}}}

\newcommand{\BoolRng}{\ensuremath{\mathsf{BoolRng}}}

\newcommand{\CAssoc}{\ensuremath{\mathsf{CAssoc}}}
\newcommand{\ACAAssoc}{\ensuremath{\mathsf{ACAAssoc}}}
\newcommand{\Com}{\ensuremath{\mathsf{Com}}}
\newcommand{\ACom}{\ensuremath{\mathsf{ACom}}}
\newcommand{\JJord}{\ensuremath{\mathsf{JJord}}}
\newcommand{\Jord}{\ensuremath{\mathsf{Jord}}}
\newcommand{\AbAlg}{\ensuremath{\mathsf{AbAlg}}}
\newcommand{\Alg}{\ensuremath{\mathsf{Alg}}}
\newcommand{\PAlg}{\ensuremath{\mathsf{PAlg}}}
\newcommand{\Nil}{\ensuremath{\mathsf{Nil}}}
\newcommand{\Sol}{\ensuremath{\mathsf{Sol}}}
\newcommand{\Alt}{\ensuremath{\mathsf{Alt}}}

\newcommand{\Grp}{\ensuremath{\mathsf{Grp}}}
\newcommand{\Ab}{\ensuremath{\mathsf{Ab}}}
\newcommand{\Vect}{\ensuremath{\mathsf{Vec}}}

\newcommand{\CRng}{\ensuremath{\mathsf{CRng}}}

\newcommand{\SplExtX}{\ensuremath{\mathsf{SplExt}(X)}}

\date{}

\begin{document}

\title[Actions in varieties of non-associative algebras: a survey]{Actions in varieties of non-associative algebras: a survey}

\author[M.~Mancini]{Manuel Mancini~\orcidlink{0000-0003-2142-6193}}

\email{manuel.mancini@uclouvain.be; manuel.mancini@unipa.it}

\address{Institut de Recherche en Mathématique et Physique, Université catholique de Louvain, chemin du cyclotron 2 bte L7.01.02, B--1348 Louvain-la-Neuve, Belgium.}
\address{Dipartimento di Matematica e Informatica, Università degli Studi di Palermo, Via Archirafi 34, 90123 Palermo, Italy.}

\begin{abstract}
This article surveys the main results on the representability of internal actions in varieties of non-associative algebras over a field. We begin by recalling the general framework of internal actions and split extensions in semi-abelian categories, before specialising to Orzech categories of interest and varieties of non-associative algebras. We discuss action representability, weak action representability, and action accessibility, with emphasis on the relationship between the latter two notions. We then describe the constructions of the universal strict general actor in Orzech categories of interest and the external weak actor in action accessible, operadic varieties of non-associative algebras over a field, as well as their role in the study of weak action representability. We also give a tensor-domain formulation of partial algebras and adapt the construction of the external weak actor to this setting. These notions are illustrated by examples arising from associative, Lie, Leibniz, Jacobi--Jordan, anti-commutative anti-associative, alternative, and $2$-nilpotent algebras. We conclude with a summary of the known results and some open questions for future research.
\end{abstract}

\subjclass[2020]{08A35; 08A55; 08C05; 16B50; 17-02; 17A01; 17A32; 17A36; 18-02; 18E13}
\keywords{Internal action, split extension, weakly action representable category, category of interest, non-associative algebra, external weak actor}

\maketitle

\section{Introduction}\label{SecIntro}

Semi-abelian categories were introduced by G.~Janelidze, L.~Márki
and W.~Tholen in~\cite{Semi-Ab} as a categorical framework capturing
many of the algebraic and homological properties shared by groups,
rings and algebras; see also the monograph~\cite{BB}. Within this framework, building on the categorical notion of semidirect product due to D.~Bourn and G.~Janelidze~\cite{BJ98}, F.~Borceux,
G.~Janelidze and G.~M.~Kelly introduced the notion of an
\emph{internal action}~\cite{BJK}, extending the classical
correspondence between actions and split extensions from groups and
Lie algebras to arbitrary semi-abelian categories. More precisely, for every object $X$ of a
semi-abelian category $\cC$, they proved that there is a natural isomorphism
\[
\Act(-,X)\cong\SplExt(-,X),
\]
where $\Act(B,X)$ denotes the set of actions of $B$ on $X$ and
$\SplExt(B,X)$ denotes the set of isomorphism classes of split
extensions of $B$ by $X$.

Internal actions are particularly well behaved when the functor
$\Act(-,X)$, or equivalently $\SplExt(-,X)$, is representable for
every object $X$. More precisely, a semi-abelian category $\cC$ is
said to be \emph{action representable} if, for every object $X$ of~$\cC$,
there exists an object $[X]$ of $\cC$, called the \emph{actor} of
$X$, together with a natural isomorphism
\[
\SplExt(-,X)\cong\Hom_{\cC}(-,[X]).
\]
The classical examples are the category of groups, where the actor
of a group~$X$ is the automorphism group $\Aut(X)$, and the category
of Lie algebras over a field, where the actor of a Lie algebra $X$ is the Lie
algebra of derivations $\Der(X)$. However, action representability is
a restrictive condition. For instance, commutative associative
algebras are not action representable~\cite{BJK2}, and it was proved in~\cite{Tim} that, over an infinite field of characteristic
different from~$2$, the only non-trivial action representable varieties of
non-associative algebras are the variety of Lie algebras and that of abelian algebras.

The search for universal acting objects, analogous to the automorphism group of a group, goes back to classical work in algebra. Bimultiplications of associative algebras and of rings were studied by G.~Hochschild~\cite{Hochschild} and by S.~Mac~Lane~\cite{MacLane58}, respectively, and actions of associative algebras by A.~S.-T.~Lue~\cite{Lue}. The term \emph{actor} goes back to the work of K.~Norrie on crossed modules of groups~\cite{Norrie}, and J.-L.~Loday introduced the biderivations of a Leibniz algebra~\cite{loday}. In a semi-abelian category, the actor of an object $X$ corresponds to a \emph{split extension classifier} for~$X$; see~\cite{BBSEC,BJK}.

The strength of action representability led to the introduction of several weaker conditions. D.~Bourn and G.~Janelidze introduced the notion of an \emph{action accessible category} in~\cite{act_accessible}, thereby including relevant examples such as rings and associative algebras. Its close relationship with the existence of centralisers and normalisers was further investigated in~\cite{BournTwoWays,CigoliMantovani,GrayNormalizers}. A.~Montoli later proved in~\cite{Montoli} that every Orzech category of interest~\cite{Orzech}, such as the category of Leibniz algebras, is action accessible. On the other hand, J.~M.~Casas, T.~Datuashvili and M.~Ladra constructed, for every object $X$ of an Orzech category of interest, a
\emph{universal strict general actor} $\USGA(X)$ in a suitable
larger category~\cite{Casas}; see also~\cite{DatuashviliSurvey}.

A different weakening was introduced by G.~Janelidze
in~\cite{WRA}. Instead of requiring a natural isomorphism, one asks,
for every object $X$ of $\cC$, for an object $T=T(X)$ of $\cC$ and a
natural monomorphism of functors
\[
\tau\colon
\SplExt(-,X)
\rightarrowtail
\Hom_{\cC}(-,T).
\]
The pair $(T,\tau)$ is called a \emph{weak representation} of
$\SplExt(-,X)$, and $T$ is called a \emph{weak representing object}
of $X$. The category $\cC$ is said to be \emph{weakly action
representable} if $\SplExt(-,X)$ admits a weak representation for
every object $X$ of $\cC$.

Every action representable category is clearly weakly action
representable. Moreover, G.~Janelidze proved in~\cite{WRA} that every
weakly action representable category is action
accessible. Thus, one has the implications
\[
\textit{action representability}
\Rightarrow
\textit{weak action representability}
\Rightarrow
\textit{action accessibility}.
\]
Nevertheless, action
accessibility does not imply weak action representability in general,
as shown by J.~R.~A.~Gray in~\cite{Gray}. More recently,
it was proved in~\cite{XabiMancini} that the varieties of $k$-nilpotent and $n$-solvable Lie algebras over a field are action accessible but not
weakly action representable, for every $k\geq3$ and $n\geq2$,
respectively. These provide the first examples of action accessible varieties of non-associative algebras over a field which fail to be weakly action representable. Moreover, it was shown in~\cite{WRAAlgCorr} that the variety $\CAssoc$ of commutative associative algebras, which is action accessible, is not weakly action representable either; this corrects~\cite[Theorem~2.11]{WRAAlg}.

For an Orzech category of interest $\cC$, if $\USGA(X)$ is an object of $\cC$ for every object $X$ of $\cC$, then the category is weakly action
representable~\cite{CigoliManciniMetere}. This applies, for example,
to the categories of associative algebras and of Leibniz algebras over a field, with suitable presentations as Orzech categories of interest; see \Cref{USGA_Assoc,USGA_Leib}.

The aim of this survey is to present a unified account of these
notions in the context of varieties of non-associative algebras over
a field. Particular attention is devoted to the construction
introduced in~\cite{WRAAlg}: for every object $X$ of an action
accessible, operadic variety of non-associative algebras $\V$, one can define a partial algebra~$\E(X)$, called the \emph{external weak actor} of $X$, together with a natural monomorphism of functors
\[
\tau\colon
\SplExt(-,X)
\rightarrowtail
\Hom_{\PAlg}(\widetilde U(-),\E(X)),
\]
where $\PAlg$ is the category of partial algebras and
$\widetilde U\colon\V\to\PAlg$ is the forgetful functor. If $\E(X)$ is a total algebra and its multiplication satisfies the defining identities of $\V$, then $\tau$ gives a weak representation of $\SplExt(-,X)$. This is the case, for instance, for the varieties of $2$-nilpotent commutative and anti-commutative algebras.

In this survey, we modify the formulation of partial algebras used in~\cite{WRAAlg}. There, the domain of a partial multiplication on a vector space~$X$ was required to be a vector subspace of $X\times X$. Since products with zero are always defined in the construction of the external weak actor, this requirement forces the corresponding multiplications to be total. We therefore take the domain to be a vector subspace of $X\otimes X$ instead: a \emph{partial algebra} is a triple $(X,D,\mu)$, where $D\leq X\otimes X$ is a vector subspace and $\mu\colon D\to X$ is a linear map, and the product $xy$ is defined if and only if $x\otimes y\in D$, in which case $xy=\mu(x\otimes y)$. A partial algebra is \emph{total} if $D=X\otimes X$, and total algebras are precisely the non-associative algebras~$X$, viewed as $\widetilde U(X)=(X,X\otimes X,\widetilde m_X)$, with $\widetilde m_X$ the linearisation of the bilinear multiplication of~$X$. This yields the forgetful functor $\widetilde U\colon\V\to\PAlg$ considered above, which is fully faithful. Accordingly, the domain of the partial multiplication of $\E(X)$ is the inverse image of $\E(X)$ under the linearisation of the bilinear map $m_X\colon\E(X)\times\E(X)\to\End(X)^2$ of \Cref{E(X)}. Thus, the product of $f,g\in\E(X)$ is defined precisely when $m_X(f,g)\in\E(X)$, and a linear map $\varphi\colon B\to\E(X)$ defines a homomorphism of partial algebras $\widetilde U(B)\to\E(X)$ if and only if $\varphi(bb')=m_X\bigl(\varphi(b),\varphi(b')\bigr)$ for all $b,b'\in B$; see \Cref{lem:morphism-from-total}. We also describe how polynomial identities are evaluated in a partial algebra; see \Cref{SecPart}.

The paper is organised as follows. In \Cref{SecIntAct}, we recall
internal actions and split extensions in semi-abelian categories and
introduce action representability, action accessibility and weak
action representability. In \Cref{SecOrz}, we describe derived
actions in Orzech categories of interest and recall the construction
of the universal strict general actor. In \Cref{SecVar}, we introduce
varieties of non-associative algebras over a field, describe their derived actions and recall the categorical characterisations relevant to the study of their representability. We also introduce the category $\PAlg$ of partial algebras. In \Cref{SecRepVar}, we construct the external weak actor and discuss its main properties and examples. Finally, in \Cref{SecConcl}, we summarise the known results in \Cref{tab:summary} and collect some open questions and possible directions for future research.

This survey is based on selected material from the author's doctoral
thesis~\cite{Mancini_Thesis}, reorganised and updated in light of
subsequent developments. It brings together results
from~\cite{BJK,BJK2,Casas,WRA,Orzech}, as well as more recent results
from~\cite{WRAAlg,WRAAlgCorr,CigoliManciniMetere,XabiMancini}. For completeness,
we give detailed arguments for some of the results recalled here;
see, for instance, \Cref{WRAimpliesAA,intderivedact}.

This survey was used as the basis for two seminars given by the author in the Informalgebra seminar series at the University of Salento, Lecce, in September 2026, and it has benefited from discussions with the participants.

\section{Internal actions in semi-abelian categories}\label{SecIntAct}

Semi-abelian categories were introduced by G.~Janelidze, L.~Márki,
and W.~Tholen in~\cite{Semi-Ab} to provide a categorical setting
capturing common categorical-algebraic properties of groups, rings,
and algebras. Recall that a category $\cC$ is semi-abelian if it is
pointed, Barr-exact~\cite{Barr}, protomodular~\cite{Bourn}, and
has finite coproducts. We refer to~\cite{BB} for a comprehensive account of Mal'tsev, protomodular, homological and semi-abelian categories.

Throughout the text, unless otherwise stated, $\cC$ denotes a semi-abelian category, and $0$ denotes its zero object. Whenever the objects of $\cC$ are algebras over a field $\bF$, we assume that $\bF$ is fixed and omit it from our notation. For any $\bF$-algebra $X$, we write $\End(X)$ for the vector space of $\bF$-linear endomorphisms of the underlying vector space of $X$.

One of the notions that can be studied in the context of semi-abelian
categories is that of an internal action~\cite{BJK}. For objects
$B$ and $X$ of a semi-abelian category~$\cC$, let
$(B\flat X,\kappa_{B,X})$ denote the kernel of the morphism
\[
[\id_B,0_{X,B}]\colon B+X\to B
\]
induced by the identity map $\id_B\colon B\to B$ and the zero morphism $0_{X,B}\colon X\to B$, where $B+X$ denotes the
\emph{coproduct} of $B$ and $X$. 

For a fixed object $B$ of $\cC$, we define the functor
\[
B\flat(-)\colon\cC\to\cC\colon X\mapsto B\flat X.
\]
We also define the natural transformations
\[
\eta^B\colon\id_{\cC}\to B\flat(-)
\qquad\text{and}\qquad
\mu^B\colon B\flat(B\flat(-))\to B\flat(-)
\]
as follows: for every object $X$ of $\cC$, the morphism
$\eta^B_X\colon X\to B\flat X$ is induced by the coproduct injection
$X\rightarrowtail B+X$. Moreover,
\[
\mu^B_X\colon B\flat(B\flat X)\to B\flat X
\]
is a restriction of the codiagonal
\[
(B+B)+X\to B+X.
\]
These data define a monad $(B\flat(-),\eta^B,\mu^B)$ on~$\cC$.

\begin{definition}\cite{BJK}\label{defact}
Let $\cC$ be a semi-abelian category and let $B$ be an object of~$\cC$. An \emph{internal action} of $B$ on an object $X$ of $\cC$ is a morphism $\xi\colon B\flat X\to X$ such that $(X,\xi)$ is an algebra over the monad $B\flat(-)$, that is, such that the diagrams
\begin{equation*}
\begin{tikzcd}[baseline=(current bounding box.center)]
X
    \arrow[r, "\eta^{B}_{X}"]
    \arrow[dr, equal]
&
B\flat X
    \arrow[d, "\xi"]
\\
&
X
\end{tikzcd}
\qquad\text{and}\qquad
\begin{tikzcd}[baseline=(current bounding box.center)]
B\flat(B\flat X)
    \arrow[r, "\mu^{B}_{X}"]
    \arrow[d, "\id_B\flat\xi"']
&
B\flat X
    \arrow[d, "\xi"]
\\
B\flat X
    \arrow[r, "\xi"']
&
X
\end{tikzcd}
\end{equation*}
commute.
\end{definition}

We write $\Act_{\cC}(B,X)$ for the set of internal actions of $B$
on $X$. When the category $\cC$ is clear from the context, we simply
write $\Act(B,X)$.

Fix an object $X$ of $\cC$. Internal actions on $X$ give rise to a functor
\[
\Act(-,X)\colon\cC^{\op}\to\Set
\]
which assigns to each object $B$ of $\cC$ the set $\Act(B,X)$. Given
an arrow $f\colon B'\to B$, the map $\Act(f,X)$ sends an internal action
$\xi\colon B\flat X\to X$ to
\[
\Act(f,X)(\xi)
\coloneqq
\xi\circ(f\flat X)\colon B'\flat X\to X,
\]
where $f\flat X\colon B'\flat X\to B\flat X$ is the unique morphism
in $\cC$ satisfying
\[
\kappa_{B,X}\circ(f\flat X)
=
(f+X)\circ\kappa_{B',X}.
\]
Here, $f+X\colon B'+X\to B+X$ is induced by $f$ on $B'$ and by
$\id_X$ on $X$.

An equivalent way of describing internal actions is in terms of
\emph{split extensions}. Recall from~\cite{BJK} that a
\emph{split extension} of $B$ by $X$ is a diagram in $\cC$
\begin{equation}\label{eq:split_ext}
\begin{tikzcd}
X \arrow[r, "k"]
&
A \arrow[r, shift left, "\alpha"]
&
B \arrow[l, shift left, "\beta"]
\end{tikzcd}
\end{equation}
such that $\alpha\circ\beta=\id_B$ and $(X,k)$ is the kernel of
$\alpha$. Protomodularity implies that the pair $(k,\beta)$ is
\emph{jointly strongly epimorphic}\footnote{Recall
from~\cite{acc} that a \emph{cospan} $(f,g)$ over an object $Z$ in a
category $\cC$, with $f\colon X\to Z$ and $g\colon Y\to Z$, is said to be \emph{jointly strongly epimorphic}, or
\emph{jointly strongly epic}, if, for every diagram
\[\xymatrix{
			& M \ar@{->}[d]_-{m}\\
			&P\\
			X \ar[r]_-{f} \ar@/^3ex/[ruu]^-{f'} & Z\ar[u]^-{\phi} \ar@{-->}@/_3ex/[uu]_-{\exists!\,\varphi} & Y \ar@/_3ex/[uul]_-{g'} \ar[l]^-{g}}
\]
in which $m$ is a monomorphism and $m\circ f'=\phi\circ f$, $m\circ g'=\phi\circ g$, there exists a unique morphism $\varphi\colon Z\to M$ such that $m\circ\varphi=\phi$. Since $m$ is a monomorphism, one then also has $\varphi\circ f=f'$ and $\varphi\circ g=g'$.}. Moreover, $\alpha$ is the cokernel of $k$.

A morphism of split extensions from
\begin{equation*}
\begin{tikzcd}
X \arrow[r, "k"]
&
A \arrow[r, shift left, "\alpha"]
&
B \arrow[l, shift left, "\beta"] 
\end{tikzcd}
\end{equation*}
to
\begin{equation*}
\begin{tikzcd}
X' \arrow[r, "k'"]
&
A' \arrow[r, shift left, "\alpha'"]
&
B' \arrow[l, shift left, "\beta'"] 
\end{tikzcd}
\end{equation*}
is a triple of morphisms in $\cC$
\[
\bigl(f\colon X\to X',\ \theta\colon A\to A',\
g\colon B\to B'\bigr)
\]
such that the following diagram commutes:
\begin{equation*}
\begin{tikzcd}
X
    \arrow[r, "k"]
    \arrow[d, "f"']
&
A
    \arrow[r, shift left, "\alpha"]
    \arrow[d, "\theta"']
&
B
    \arrow[l, shift left, "\beta"]
    \arrow[d, "g"]
\\
X' \arrow[r, "k'"]
&
A' \arrow[r, shift left, "\alpha'"]
&
B' \arrow[l, shift left, "\beta'"].
\end{tikzcd}
\end{equation*}
Since $(k,\beta)$ is jointly strongly epimorphic, $\theta$ is uniquely
determined by $f$ and $g$. Moreover, by protomodularity, every
morphism of split extensions fixing $X$ and $B$ is necessarily an
isomorphism.

For an object $X$ of $\cC$, we define the functor
\[
\SplExt_{\cC}(-,X)\colon\cC^{\op}\to\Set
\]
which assigns to each object $B$ of $\cC$ the set
$\SplExt_{\cC}(B,X)$ of isomorphism classes of split extensions of
$B$ by $X$ in $\cC$. To a morphism $f\colon B'\to B$, it assigns the
\emph{change of base} function
\[
f^*\coloneqq\SplExt_{\cC}(f,X)\colon
\SplExt_{\cC}(B,X)\to\SplExt_{\cC}(B',X),
\]
defined by pulling back along $f$.

More precisely, given a split extension~\eqref{eq:split_ext}, one
forms the pullback of $\alpha$ along $f\colon B'\to B$, obtaining the
following diagram:
\[
\begin{tikzcd}
X
    \arrow[r, "k'"]
    \arrow[d, equal]
&
A\times_B B'
    \arrow[r, shift left, "p_2"]
    \arrow[d, "p_1"']
&
B'
    \arrow[l, shift left, "\beta'"]
    \arrow[d, "f"]
\\
X \arrow[r, "k"]
&
A \arrow[r, shift left, "\alpha"]
&
B \arrow[l, shift left, "\beta"],
\end{tikzcd}
\]
where $p_1$ and $p_2$ are the pullback projections, and
\[
\beta'=\langle\beta\circ f,\id_{B'}\rangle
\]
is the unique section of $p_2$ induced by the universal property of
the pullback. Notice that $p_2$ has the same kernel as $\alpha$, with $k'=\langle k,0_{X,B'}\rangle$.

When the category is clear from the context, we simply write $\SplExt(-,X)$.

The relationship between internal actions and split extensions in a
semi-abelian category is described by the following result.

\begin{theorem}\cite{BJK}\label{actspltnatural}
For every object $X$ of $\cC$, there is a natural isomorphism of
functors
\[
\Act(-,X)\cong\SplExt(-,X).
\]
\end{theorem}

\subsection{Action representable and action accessible categories}
\label{SecAR}

In~\cite{BJK}, F.~Borceux, G.~Janelidze, and G.~M.~Kelly introduced
the notion of an \emph{action representable} category.

\begin{definition}\cite{BJK}
A semi-abelian category $\cC$ is said to be \emph{action representable}
if, for every object $X$ of $\cC$, the functor $\SplExt(-,X)$ is
representable. In other words, there exists an object $[X]$ of $\cC$,
called the \emph{actor} of $X$, together with a natural isomorphism of
functors
\begin{align*}
\SplExt(-,X)\cong\Hom_{\cC}(-,[X]).
\end{align*}
\end{definition}

Basic examples of action representable categories are the category
$\Grp$ of groups, the category $\Lie$ of Lie algebras over a field,
and any abelian category. The category of not necessarily unital Boolean rings is action representable as well~\cite{BJK2}. The relationship between action representability, normalisers and centralisers is studied in~\cite{GrayAR}.

\begin{example}
Recall that an (external) action of a group $B$ on a group $X$ is a
map
\[
\xi\colon B\times X\to X\colon (b,x)\mapsto b\cdot x
\]
such that $1_B\cdot x=x$, $(bb')\cdot x=b\cdot(b'\cdot x)$, and
$b\cdot(xx')=(b\cdot x)(b\cdot x')$, for all $b,b'\in B$ and
$x,x'\in X$. Here, we use multiplicative notation for both $B$ and
$X$.

Given an action $\xi\colon B\times X\to X$ in $\Grp$, one can
construct the \emph{semidirect product} $B\ltimes X$ of $B$ and
$X$ with respect to $\xi$. Its underlying set is the Cartesian product
$B\times X$, and its multiplication is given by
\[
(b,x)(b',x')=(bb',x(b\cdot x')),
\]
with $1_{B\ltimes X}=(1_B,1_X)$ and
$(b,x)^{-1}=(b^{-1},b^{-1}\cdot x^{-1})$. Thus, we can associate with~$\xi$ the following split extension in $\Grp$:
\begin{equation}\label{splextgrpsemidirect}
\begin{tikzcd}
X \arrow[r, "\iota_2"]
&
B\ltimes X \arrow[r, shift left, "\pi_1"]
&
B \arrow[l, shift left, "\iota_1"],
\end{tikzcd}
\end{equation}
where $\iota_2(x)=(1_B,x)$, $\iota_1(b)=(b,1_X)$, and $\pi_1(b,x)=b$.

Conversely, suppose that we start with a split extension in $\Grp$
\begin{equation}\label{splextgrp}
\begin{tikzcd}
X \arrow[r, "k"]
&
A \arrow[r, shift left, "\alpha"]
&
B \arrow[l, shift left, "\beta"].
\end{tikzcd}
\end{equation}
We can then define an action of $B$ on $X$ by\footnote{Indeed, one has
\[
\alpha\bigl(\beta(b)k(x)\beta(b)^{-1}\bigr)
=\alpha(\beta(b))\alpha(k(x))\alpha(\beta(b))^{-1}
=b1_Bb^{-1}=1_B.
\]
Hence, $\beta(b)k(x)\beta(b)^{-1}\in\ker(\alpha)\cong X$, so that
$k^{-1}\bigl(\beta(b)k(x)\beta(b)^{-1}\bigr)$ is well defined.}
\begin{equation}\label{actgrp}
B\times X\to X\colon
(b,x)\mapsto
k^{-1}\bigl(\beta(b)k(x)\beta(b)^{-1}\bigr)
\end{equation}
and there is an isomorphism of split extensions
\begin{equation*}
\begin{tikzcd}
X
    \arrow[r, "\iota_2"]
    \arrow[d, equal]
&
B\ltimes X
    \arrow[r, shift left, "\pi_1"]
    \arrow[d, "\theta"']
&
B
    \arrow[l, shift left, "\iota_1"]
    \arrow[d, equal]
\\
X \arrow[r, "k"]
&
A \arrow[r, shift left, "\alpha"]
&
B \arrow[l, shift left, "\beta"],
\end{tikzcd}
\end{equation*}
where $B\ltimes X$ is the semidirect product associated with
\eqref{actgrp} and $\theta(b,x)=k(x)\beta(b)$.

Moreover, isomorphism classes of split extensions of groups are in bijection with group homomorphisms $B\to\Aut(X)$, where $\Aut(X)$ denotes the group of \emph{automorphisms} of $X$. Given a split extension of groups as in \eqref{splextgrp}, one can define a group homomorphism
\[
\varphi\colon B\to\Aut(X)\colon b\mapsto\varphi_b,
\]
where
$\varphi_b(x)=k^{-1}\bigl(\beta(b)k(x)\beta(b)^{-1}\bigr)$.
Conversely, every morphism $\varphi\colon B\to\Aut(X)$ defines a
semidirect product $B\ltimes X$ with multiplication
\[
(b,x)(b',x')=(bb',x\varphi_b(x')),
\]
and hence a split extension of $B$ by $X$ as in
\eqref{splextgrpsemidirect}.

This bijection is natural in $B$, since the diagram in $\Set$
\[
\begin{tikzcd}
\SplExt(B,X)
    \arrow[r]
    \arrow[d, swap, "{f^*}"]
&
\Hom_\Grp(B,\Aut(X))
    \arrow[d, "{-\circ f}"]
\\
\SplExt(B',X)
    \arrow[r]
&
\Hom_\Grp(B',\Aut(X))
\end{tikzcd}
\]
commutes for every group homomorphism $f\colon B'\to B$. Thus,
$\Aut(X)$ is the actor of $X$, and the category $\Grp$ is action
representable.
\end{example}

\begin{example}\label{exliealgebras}
The same argument applies to the category of Lie algebras over a field,
with $\Aut(X)$ replaced by the Lie algebra of derivations $\Der(X)$.

Let $\bF$ be a field. Recall that a Lie algebra $(X,[-,-])$ is a
vector space $X$ over $\bF$ endowed with a bilinear map
\[
[-,-]\colon X\times X\to X,
\]
called \emph{commutator} or \emph{Lie bracket}, which is
\emph{alternating}, that is, $[x,x]=0$ for any~$x\in X$, and
satisfies the \emph{Jacobi identity}
\[
[x,[y,z]]+[y,[z,x]]+[z,[x,y]]=0
\]
for all $x,y,z\in X$. We denote by $\Lie$ the category of Lie algebras
over $\bF$, whose morphisms are the $\bF$-linear maps preserving the
bracket.

An (external) action of a Lie algebra $B$ on a Lie algebra $X$ is
given by a pair of bilinear maps
\[
l\colon B\times X\to X,
\qquad
r\colon X\times B\to X,
\]
denoted by $[b,x]=l(b,x)$ and $[x,b]=r(x,b)$, such that
\begin{itemize}
\item[(i)] $[x,b]=-[b,x]$;
\item[(ii)]
$[b,[x,x']]=[[b,x],x']+[x,[b,x']]$;
\item[(iii)]
$[[b,b'],x]=[b,[b',x]]
-[b',[b,x]]$;
\end{itemize}
for all $b,b'\in B$ and $x,x'\in X$. The corresponding semidirect
product $B\ltimes X$ is the Lie algebra defined on the direct sum of
vector spaces $B\oplus X$ with commutator
\[
[(b,x),(b',x')]
=
([b,b'],[x,x']+[b,x']-[b',x]).
\]
Thus, we can associate with the action $(l,r)$ the following split
extension in $\Lie$:
\begin{equation}\label{splextliesemidirect}
\begin{tikzcd}
X \arrow[r, "\iota_2"]
&
B\ltimes X \arrow[r, shift left, "\pi_1"]
&
B. \arrow[l, shift left, "\iota_1"]
\end{tikzcd}
\end{equation}
Here, $B\ltimes X=(B\oplus X,[-,-])$ is the semidirect product of $B$ and $X$ with respect to the bilinear maps $(l,r)$, while $\iota_2(x)=(0,x)$, $\iota_1(b)=(b,0)$, and
$\pi_1(b,x)=b$.

Conversely, suppose that we start with a split extension in $\Lie$
\begin{equation}\label{splextlie}
\begin{tikzcd}
X \arrow[r, "k"]
&
A \arrow[r, shift left, "\alpha"]
&
B. \arrow[l, shift left, "\beta"]
\end{tikzcd}
\end{equation}
We can then define an action of $B$ on $X$ by\footnote{Again, one has
\[
\alpha\bigl([\beta(b),k(x)]\bigr)
=[\alpha(\beta(b)),\alpha(k(x))]
=[b,0]=0.
\]
Hence, $[\beta(b),k(x)]\in\ker(\alpha) \cong X$, so that
$k^{-1}\bigl([\beta(b),k(x)]\bigr)$ is well defined.}
\begin{equation}\label{actlie}
[b,x]=-[x,b]=k^{-1}\bigl([\beta(b),k(x)]\bigr),
\end{equation}
and there is an isomorphism of split extensions
\begin{equation*}
\begin{tikzcd}
X
    \arrow[r, "\iota_2"]
    \arrow[d, equal]
&
B\ltimes X
    \arrow[r, shift left, "\pi_1"]
    \arrow[d, "\theta"']
&
B
    \arrow[l, shift left, "\iota_1"]
    \arrow[d, equal]
\\
X \arrow[r, "k"]
&
A \arrow[r, shift left, "\alpha"]
&
B \arrow[l, shift left, "\beta"].
\end{tikzcd}
\end{equation*}
Here, $B\ltimes X$ is the semidirect product associated with
\eqref{actlie} and $\theta(b,x)=\beta(b)+k(x)$.

We further recall that a derivation of a Lie algebra $X$ is a linear
endomorphism $d\colon X\to X$ such that
\[
d([x,y])=[d(x),y]+[x,d(y)],
\]
for all $x,y\in X$. The vector space
$\Der(X)$ of derivations of $X$ is a Lie algebra with respect to the
bracket
\[
[d,d']=d\circ d'-d'\circ d.
\]
The isomorphism classes of split extensions of Lie algebras are in bijection with the Lie algebra homomorphisms $B\to\Der(X)$. Given a split extension of Lie algebras as in~\eqref{splextlie}, one can define a Lie algebra homomorphism
\[
\varphi\colon B\to\Der(X)\colon b\mapsto\varphi_b,
\]
where $\varphi_b(x)=k^{-1}\bigl([\beta(b),k(x)]_A\bigr)$.
Conversely, every morphism $\varphi\colon B\to\Der(X)$ defines a
semidirect product $B\ltimes X$ with Lie bracket
\[
[(b,x),(b',x')]
=
([b,b'],[x,x']+\varphi_b(x')-\varphi_{b'}(x)),
\]
and hence a split extension of $B$ by $X$ as in
\eqref{splextliesemidirect}.

This bijection is natural in $B$, since the diagram in $\Set$
\[
\begin{tikzcd}
\SplExt(B,X)
    \arrow[r]
    \arrow[d, swap, "{f^*}"]
&
\Hom_\Lie(B,\Der(X))
    \arrow[d, "{-\circ f}"]
\\
\SplExt(B',X)
    \arrow[r]
&
\Hom_\Lie(B',\Der(X))
\end{tikzcd}
\]
commutes for every Lie algebra homomorphism $f\colon B'\to B$. Thus,
$\Der(X)$ is the actor of $X$, and the category $\Lie$ is action
representable.
\end{example}

A first example of a category that is not action representable is the
category $\CAssoc$ of commutative associative algebras over a field
$\bF$. In~\cite{BJK2}, the authors proved that, for every commutative
associative algebra $X$, there is a natural isomorphism of functors
\[
\SplExt_{\CAssoc}(-,X)
\cong
\Hom_{\Assoc}(\overline{U}(-),\M(X)),
\]
where $\overline{U}\colon\CAssoc\to\Assoc$ denotes the forgetful functor and
\[
\M(X)
=
\{f\in\End(X)\mid f(xy)=f(x)y, \; \forall x,y\in X\}
\]
is the associative algebra of \emph{multipliers} of
$X$~\cite{MacLane58}, endowed with the usual composition of
functions. Moreover, $\M(X)$ is a commutative algebra if and only if
the functor $\SplExt_{\CAssoc}(-,X)$ is representable
(see~\cite[Theorem~2.6]{BJK2}). Since there are commutative
associative algebras $X$ for which $\M(X)$ is not commutative—for
instance, any abelian algebra $X$ with $\dim_{\bF}X\geq2$, for which
$\M(X)=\End(X)$—the category $\CAssoc$ is not action representable.

An analogous result was then established
in~\cite[Proposition~1.11]{Tim} for the category $\Assoc$ of
associative algebras over $\bF$. Indeed, if $\Assoc$ were action
representable, then, for any associative algebras $B_1,B_2$ and $X$,
action representability would imply that
\[
\SplExt(B_1+B_2,X)
\cong
\SplExt(B_1,X)\times\SplExt(B_2,X).
\]
Thus, every pair of actions of $B_1$ and $B_2$ on $X$ would determine
an action of their coproduct $B_1+B_2$ on $X$. However, one can
construct two actions on an abelian algebra $X$, respectively by
one-dimensional abelian algebras $B_1$ and $B_2$, which cannot be
combined in this way. More precisely, the multiplication induced on
the corresponding semidirect product would fail to be associative. This contradiction shows that $\Assoc$ is not action representable.

Given an internal action $\xi$ of $B$ on $X$, we denote by $B\ltimes X$ the middle object of the split extension corresponding to $\xi$ under \Cref{actspltnatural}, and we call it the \emph{semidirect product} of $B$ and $X$ with respect to~$\xi$.

\begin{remark}
Let $\SplExtX$ be the category of split extensions in $\cC$ with fixed kernel~$X$, whose morphisms are triples
$(\id_X,\theta,g)$ making the diagrams of split extensions commute. Since a functor $\cC^{\op}\to\Set$ is representable if and only if its
category of elements has a terminal object (see~\cite{BJK}), action
representability is equivalent to requiring that, for every object
$X$ of $\cC$, the category $\SplExtX$ has a terminal object of the form
\begin{equation*}
\begin{tikzcd}
X \arrow[r, "k_X"]
&
{[X]\ltimes X} \arrow[r, shift left, "\alpha_X"]
&
{[X]} \arrow[l, shift left, "\beta_X"].
\end{tikzcd}
\end{equation*}
\end{remark}

More precisely, we have the following result.

\begin{theorem}\cite{BJK}\label{termobj}
A semi-abelian category $\cC$ is action representable if and only if,
for every object $X$ of $\cC$, there exist an object $[X]$ of $\cC$
and an internal action of~$[X]$ on $X$ with the following universal
property: for every object $B$ of $\cC$ and every split extension
\begin{equation*}
\begin{tikzcd}
X \arrow[r, "k"]
&
A \arrow[r, shift left, "\alpha"]
&
B \arrow[l, shift left, "\beta"],
\end{tikzcd}
\end{equation*}
there is a unique morphism $b \colon B\to[X]$ such that there exists a morphism
$\theta \colon A\to[X]\ltimes X$ making the following diagram commute:
\begin{equation*} 
\begin{tikzcd}
X
    \arrow[r, "k"]
    \arrow[d, equal]
&
A
    \arrow[r, shift left, "\alpha"]
    \arrow[d, "\theta"']
&
B
    \arrow[l, shift left, "\beta"]
    \arrow[d, "b"]
\\
X \arrow[r, "k_X"]
&
{[X]\ltimes X} \arrow[r, shift left, "\alpha_X"]
&
{[X]} \arrow[l, shift left, "\beta_X"].
\end{tikzcd}
\end{equation*}
\end{theorem}

We include an argument for completeness.

\begin{proof}
Suppose first that $\cC$ is action representable. Let $X$ be an
object of $\cC$, and let $[X]$ be its actor. The identity morphism
$\id_{[X]}$ determines an internal action $\xi$ of~$[X]$ on $X$.
Denote by
\begin{equation*}
\begin{tikzcd}
X \arrow[r, "k_X"]
&
{[X]\ltimes X} \arrow[r, shift left, "\alpha_X"]
&
{[X]} \arrow[l, shift left, "\beta_X"]
\end{tikzcd}
\end{equation*}
the split extension associated with $\xi$ under the natural bijection
of \Cref{actspltnatural}.

The category of elements $\mathsf{el}(\SplExt(-,X))$ has a terminal
object given by this split extension. Thus, for every object
$B$ of $\cC$ and every split extension
\begin{equation*}
\begin{tikzcd}
X \arrow[r, "k"]
&
A \arrow[r, shift left, "\alpha"]
&
B \arrow[l, shift left, "\beta"],
\end{tikzcd}
\end{equation*}
there exist a unique morphism $b\colon B\to[X]$ and a (necessarily unique) morphism $\theta$ such that the
following diagram commutes:
\begin{equation*}
\begin{tikzcd}
X
    \arrow[r, "k"]
    \arrow[d, equal]
&
A
    \arrow[r, shift left, "\alpha"]
    \arrow[d, "\theta"']
&
B
    \arrow[l, shift left, "\beta"]
    \arrow[d, "b"]
\\
X \arrow[r, "k_X"]
&
{[X]\ltimes X} \arrow[r, shift left, "\alpha_X"]
&
{[X]} \arrow[l, shift left, "\beta_X"].
\end{tikzcd}
\end{equation*}
The morphism $\theta$ is uniquely determined by $\id_X$ and $b$, and the right-hand square is a pullback.

Conversely, suppose that, for every object $X$ of $\cC$, there exists
an object $[X]$ acting on $X$ and satisfying the above universal
property. For every object $B$ of $\cC$, let
\[
\tau_B\colon\SplExt(B,X)\longrightarrow\Hom_{\cC}(B,[X])
\]
send the isomorphism class of a split extension to the unique morphism
provided by the universal property. On the other hand, every morphism
$b\colon B\to[X]$ determines a split extension of $B$ by $X$ by
pulling back along $b$ the split extension associated with the action of $[X]$ on $X$. These two assignments are mutually inverse: one composite is the identity by the uniqueness in the universal property, and the other one because, by protomodularity, the right-hand square of a morphism of split extensions which is the identity on kernels is a pullback. This establishes a bijection
\[
\SplExt(B,X)\cong\Hom_{\cC}(B,[X])
\]
for every object $B$ of $\cC$, and naturality follows from the
uniqueness in the universal property.
Hence, the maps~$\tau_B$ determine a natural isomorphism
\[
\SplExt(-,X)\cong\Hom_{\cC}(-,[X]).
\]
Therefore, $[X]$ is the actor of $X$, and $\cC$ is action
representable.
\end{proof}

\begin{example}
Let $\cC=\Grp$, and let $X$ be a group. Then $\Aut(X)$ acts on $X$ via
\[
\Aut(X)\times X\to X\colon(\varphi,x)\mapsto\varphi(x).
\]
The corresponding semidirect product $\Aut(X)\ltimes X$ is called the
\emph{holomorph} of $X$ and is denoted by $\Hol(X)$;
see~\cite{Burnside}. Its multiplication is given by
\[
(\varphi,x)(\varphi',x')
=
(\varphi\circ\varphi',x\varphi(x')).
\]
\end{example}

\begin{example}\label{exactreplie}
Let $\cC=\Lie$, and let $X$ be a Lie algebra. The action of the Lie
algebra $\Der(X)$ on $X$ is defined by the pair of bilinear maps
\[
\Der(X)\times X\to X\colon(d,x)\mapsto d(x)
\]
and
\[
X\times\Der(X)\to X\colon(x,d)\mapsto-d(x).
\]
The associated semidirect product $\Der(X)\ltimes X$ is again called
the \emph{holomorph} of $X$ and is denoted by
$\Hol(X)$; see~\cite{jacobsonLie}. Its Lie bracket is
given by
\[
[(d,x),(d',x')]
=
(d\circ d'-d'\circ d,[x,x']+d(x')-d'(x)).
\]
\end{example}

We may weaken action representability by requiring instead that, for
every object~$X$ of a semi-abelian category $\cC$, every object of
$\SplExtX$ be \emph{accessible}, that is, admit a morphism into a
\emph{subterminal} or \emph{faithful} object~\cite{act_accessible}.
Such an object admits at most one morphism from any given object.

\begin{definition}\cite{act_accessible}\label{defactacc}
A semi-abelian category $\cC$ is \emph{action accessible} if, for every
object $X$ of $\cC$, every object of $\SplExtX$ is accessible.
\end{definition}

The notion of \emph{action accessibility} was introduced by D.~Bourn
and G.~Janelidze in~\cite{act_accessible} to calculate centralisers of normal subobjects and of equivalence relations. It encompasses a wider class of categories, including categories that are not action
representable, such as $\CAssoc$ and $\Assoc$; more generally, every
Orzech category of interest is action accessible; see \Cref{SecOrz}.

The relationship between action accessibility, and hence action representability, and the existence of centralisers of equivalence relations is further analysed by D.~Bourn in~\cite{BournTwoWays}, where the existence of such centralisers is characterised by a property of the fibration of points; see, in particular, \cite[Proposition~5.9]{BournTwoWays}. We also refer to~\cite{CigoliMantovani}, where action accessibility is studied via centralisers. Moreover, recall that a \emph{normaliser} of a monomorphism $m\colon S\rightarrowtail A$ is a factorisation of $m$ as a normal monomorphism $S\rightarrowtail N$ followed by a monomorphism $N\rightarrowtail A$, which is terminal among all such factorisations. J.~R.~A.~Gray proved in~\cite{GrayNormalizers} that a pointed exact protomodular category $\cC$ is action accessible if and only if, for every normal monomorphism $k\colon X\rightarrowtail A$, the monomorphism $\langle k,k\rangle\colon X\rightarrowtail A\times A$ admits a normaliser.

\begin{remark}
By definition, the existence of a terminal object in $\SplExtX$ is
stronger than the accessibility of each of its objects. It follows
immediately that
\[
\textit{action representability}
\Rightarrow
\textit{action accessibility}.
\]
\end{remark}

\subsection{Weakly action representable categories}\label{SecWRA}

More recently, G.~Janelidze introduced the notion of a \emph{weakly action representable} category in~\cite{WRA}.

\begin{definition}\cite{WRA}\label{defWRA}
A semi-abelian category $\cC$ is said to be \emph{weakly action
representable} if, for every object $X$ of $\cC$, there exist an object
$T=T(X)$ of $\cC$ and a natural monomorphism of functors
\[
\tau\colon\SplExt(-,X)\rightarrowtail\Hom_{\cC}(-,T).
\]
The pair $(T,\tau)$ is called a \emph{weak representation} of
$\SplExt(-,X)$, the object $T$ is called a \emph{weak representing
object} of $X$, and a morphism $\varphi\colon B\to T$ belonging to
the image of $\tau_B$ is called an \emph{acting morphism}.
\end{definition}

\begin{remark}
It follows from \Cref{defWRA} that a weak representing object need not
be unique. Indeed, if $(T,\tau)$ is a weak representation of
$\SplExt(-,X)$ and $i\colon T\rightarrowtail T'$ is a monomorphism in
$\cC$, then $i$ induces a natural monomorphism of functors
\[
i_*\colon
\Hom_{\cC}(-,T)
\rightarrowtail
\Hom_{\cC}(-,T'),
\]
and $(T',i_*\circ\tau)$ is another weak representation of
$\SplExt(-,X)$.
\end{remark}

\begin{example}
Every action representable category is weakly action representable. In
this case, $T=[X]$ is the actor of $X$, the natural monomorphism
$\tau$ is a natural isomorphism, and every morphism
$\varphi\colon B\to[X]$ is an acting morphism.
\end{example}

The category $\Assoc$ of associative algebras over a field $\bF$ is a
first example of a weakly action representable category that is not
action representable; see~\cite{WRA}.

\begin{remark}\label{exassoc}
Consider a split extension in $\Assoc$
\begin{equation}\label{diag:SplAssoc}
\begin{tikzcd}
X \arrow[r, "k"]
&
A \arrow[r, shift left, "\alpha"]
&
B \arrow[l, shift left, "\beta"].
\end{tikzcd}
\end{equation}
We can define a pair of bilinear maps
\[
l\colon B\times X\to X,
\qquad
r\colon X\times B\to X
\]
by
\[
b\ast x
=
l(b,x)
=
k^{-1}\bigl(\beta(b)k(x)\bigr),
\qquad
x\ast b
=
r(x,b)
=
k^{-1}\bigl(k(x)\beta(b)\bigr),
\]
for all $b\in B$ and $x\in X$\footnote{These expressions are well defined, since $\beta(b)k(x),k(x)\beta(b)\in\ker(\alpha)$, and the corestriction $k \colon X \to \ker(\alpha)$ is an isomorphism.}. Moreover, we can endow the direct sum
of vector spaces~$B\oplus X$ with the bilinear multiplication
\begin{equation}\label{product on semi direct assoc}
(b,x)\cdot(b',x')
=
(bb',xx'+b\ast x'+x\ast b').
\end{equation}
Then $(B\oplus X,\cdot)$ is an associative algebra, called the
\emph{semidirect product} of $B$ and $X$ with respect to $(l,r)$ and
denoted by $B\ltimes X$. The diagram
\begin{equation*}
\begin{tikzcd}
X \arrow[r, "\iota_2"]
&
B\ltimes X \arrow[r, shift left, "\pi_1"]
&
B \arrow[l, shift left, "\iota_1"]
\end{tikzcd}
\end{equation*}
with $\iota_2(x)=(0,x)$, $\iota_1(b)=(b,0)$, and
$\pi_1(b,x)=b$, defines a split extension in $\Assoc$ which is
isomorphic to~\eqref{diag:SplAssoc} via
\begin{equation*}
\begin{tikzcd}
X
    \arrow[r, "\iota_2"]
    \arrow[d, equal]
&
B\ltimes X
    \arrow[r, shift left, "\pi_1"]
    \arrow[d, "\theta"']
&
B
    \arrow[l, shift left, "\iota_1"]
    \arrow[d, equal]
\\
X \arrow[r, "k"]
&
A \arrow[r, shift left, "\alpha"]
&
B \arrow[l, shift left, "\beta"].
\end{tikzcd}
\end{equation*}
Here, $\theta(b,x)=\beta(b)+k(x)$ for every $(b,x)\in B\oplus X$.

The associativity of the multiplication
in~\eqref{product on semi direct assoc} implies the following
identities:
\begin{itemize}
\item[(i)] $b\ast(xx')=(b\ast x)x'$;
\item[(ii)] $(xx')\ast b=x(x'\ast b)$;
\item[(iii)] $x(b\ast x')=(x\ast b)x'$;
\item[(iv)] $(bb')\ast x=b\ast(b'\ast x)$;
\item[(v)] $x\ast(bb')=(x\ast b)\ast b'$;
\item[(vi)] $b\ast(x\ast b')=(b\ast x)\ast b'$;
\end{itemize}
for all $b,b'\in B$ and $x,x'\in X$. The first three identities state
that the pair $(b\ast-,-\ast b)$ belongs to the associative algebra
\[
\Bim(X)
=
\left\{
(f\ast-,-\ast f)
\in
\End(X)\times\End(X)^{\op}
\,\middle|\,
\begin{aligned}
f\ast(xy) &= (f\ast x)y,\\
(xy)\ast f &= x(y\ast f),\\
x(f\ast y) &= (x\ast f)y
\end{aligned}
\, \forall x,y\in X
\right\}
\]
of \emph{bimultipliers} of $X$, also called \emph{bimultiplications}~\cite{MacLane58}. Its bilinear multiplication is given by
\[
(f\ast-,-\ast f)\cdot(g\ast-,-\ast g)
=
\bigl(f\ast(g\ast-),(-\ast f)\ast g\bigr).
\]
Identities~(iv) and~(v) imply that the linear map
\[
B\to\Bim(X)\colon b\mapsto(b\ast-,-\ast b)
\]
is a homomorphism of associative algebras. Moreover, this homomorphism
satisfies~(vi), which states that the left multiplier $b\ast-$ and
the right multiplier $-\ast b'$ are permutable.

Conversely, let
\[
\varphi\colon B\to\Bim(X)\colon
b\mapsto(b\ast_{\varphi}-,-\ast_{\varphi}b)
\]
be an associative algebra homomorphism satisfying
$b\ast_{\varphi}(x\ast_{\varphi}b')
=(b\ast_{\varphi}x)\ast_{\varphi}b'$. We can associate with
$\varphi$ the split extension in $\Assoc$
\begin{equation*}
\begin{tikzcd}
X \arrow[r, "\iota_2"]
&
{(B\oplus X,\cdot_{\varphi})}
    \arrow[r, shift left, "\pi_1"]
&
B, \arrow[l, shift left, "\iota_1"]
\end{tikzcd}
\end{equation*}
where the associative algebra structure on $B\oplus X$ is given by
\[
(b,x)\cdot_{\varphi}(b',x')
=
(bb',xx'+b\ast_{\varphi}x'+x\ast_{\varphi}b').
\]
As a consequence, we obtain the following result.
\end{remark}

\begin{theorem}\label{thmassoc}\cite{CigoliManciniMetere, WRA}
Let $B$ and $X$ be associative algebras over $\bF$.
\begin{itemize}
\item[(1)] The set of isomorphism classes of split extensions of $B$
by $X$ is in bijection with the set of associative algebra
homomorphisms
\[
B\to\Bim(X)\colon b\mapsto(b\ast-,-\ast b)
\]
satisfying
\begin{equation}\label{condition_assoc}
b\ast(x\ast b')=(b\ast x)\ast b'
\end{equation}
for all $b,b'\in B$ and $x\in X$.

\item[(2)] The category $\Assoc$ of associative algebras over $\bF$
is weakly action representable, and $\Bim(X)$
is a weak representing object of $X$.
\end{itemize}
\end{theorem}

\begin{remark}\label{exnotacting}
A general homomorphism $B\to\Bim(X)$ need not give rise to a split
extension in $\Assoc$. For instance, let $B=\bF\{b\}$ and
$X=\bF\{e_1,e_2\}$ be abelian algebras, where $\bF\{S\}$ denotes the vector space with basis~$S$. Then
\[
\Bim(X)=\End(X)\times\End(X)^{\op},
\]
and the homomorphism
$\varphi\colon B\to\Bim(X) \colon b \mapsto (b \ast_\varphi -, - \ast_\varphi b)$ defined by
\[
b\ast_{\varphi}e_1=e_2,
\quad
b\ast_{\varphi}e_2=0,
\qquad
e_1\ast_{\varphi}b=0,
\quad
e_2\ast_{\varphi}b=e_1
\]
does not satisfy~(vi), since
\[
(b\ast_{\varphi}e_1)\ast_{\varphi}b
=
e_2\ast_{\varphi}b
=
e_1
\neq
0
=
b\ast_{\varphi}0
=
b\ast_{\varphi}(e_1\ast_{\varphi}b).
\]
Condition~\eqref{condition_assoc} can nevertheless be used to
characterise the acting morphisms in $\Assoc$. In~\cite{MacLane58},
S.~Mac Lane described, for a ring $\Lambda$, the
$\Lambda$-bimodule structures on an abelian group $K$ in terms of ring
homomorphisms from $\Lambda$ to the ring of bimultipliers of $K$. The
following result is a straightforward generalisation to actions on an
object that is not necessarily abelian.
\end{remark}

\begin{proposition}\cite{CigoliManciniMetere}\label{actmorpAssoc}
Let $B$ and $X$ be associative algebras over $\bF$, and let
\[
\varphi\colon B\to\Bim(X)\colon
b\mapsto(b\ast_{\varphi}-,-\ast_{\varphi}b)
\]
be a morphism in $\Assoc$. Then $\varphi$ is an \emph{acting morphism}
if and only if
\begin{equation}\label{eqAssoc2}
b\ast_{\varphi}(x\ast_{\varphi}b')
=
(b\ast_{\varphi}x)\ast_{\varphi}b'
\end{equation}
for all $b,b'\in B$ and $x\in X$.
\end{proposition}

As observed in \Cref{exnotacting}, not every homomorphism
$B\to\Bim(X)$ is an acting morphism. Nevertheless, there are
associative algebras $X$ for which condition~\eqref{eqAssoc2} is
automatically satisfied~\cite{Casas}. This is the case, for instance,
when the \emph{annihilator}
\[
\Ann(X)=\{x\in X\mid xy=yx=0,\, \forall y\in X\}
\]
is trivial, or when $X$ is \emph{perfect}, that is, $X^2=X$, where
$X^2$ denotes the subspace spanned by $\{xy\mid x,y\in X\}$. Under
either assumption, one has
\begin{equation}\label{eqAssoc}
f\ast(x\ast f')=(f\ast x)\ast f'
\end{equation}
for all
$(f\ast-,-\ast f),(f'\ast-,-\ast f')\in\Bim(X)$ and $x\in X$.

Indeed, suppose first that $\Ann(X)=0$ and set
\[
z=f\ast(x\ast f')-(f\ast x)\ast f'.
\]
For every $y\in X$, the defining identities of bimultipliers give
\begin{align*}
\bigl(f\ast(x\ast f')\bigr)y
 &=f\ast\bigl((x\ast f')y\bigr)\\
 &=f\ast\bigl(x(f'\ast y)\bigr)\\
 &=(f\ast x)(f'\ast y)\\
 &=\bigl((f\ast x)\ast f'\bigr)y.
\end{align*}
Thus, $zy=0$ for every $y\in X$. A similar computation shows that
$yz=0$ for every~$y\in X$, and hence $z\in\Ann(X)=0$. Therefore,
\eqref{eqAssoc} holds.

Suppose now that $X$ is perfect. For all $x,y\in X$, one has
\begin{align*}
f\ast\bigl((xy)\ast f'\bigr)
 &=f\ast\bigl(x(y\ast f')\bigr)\\
 &=(f\ast x)(y\ast f')\\
 &=\bigl((f\ast x)y\bigr)\ast f'\\
 &=\bigl(f\ast(xy)\bigr)\ast f'.
\end{align*}
Since $X$ is spanned by products, bilinearity implies that
\eqref{eqAssoc} holds for every element of $X$.

Thus, in either case, for every associative algebra $B$, each
homomorphism
\[
\varphi\colon B\to\Bim(X)
\]
is an acting morphism. Hence, there is a natural isomorphism of
functors
\[
\SplExt(-,X)
\cong
\Hom_{\Assoc}(-,\Bim(X)),
\]
so that $\Bim(X)$ is the actor of $X$.

\begin{remark}\cite{WRA}
Let $\cC$ be a weakly action representable category. It follows from
\Cref{defWRA} that, given two objects of $\SplExtX$ and a morphism
between them
\begin{equation*}
\begin{tikzcd}
X
    \arrow[r, "k"]
    \arrow[d, equal]
&
A
    \arrow[r, shift left, "\alpha"]
    \arrow[d, "\theta"']
&
B
    \arrow[l, shift left, "\beta"]
    \arrow[d, "f"]
\\
X \arrow[r, "k'"]
&
A' \arrow[r, shift left, "\alpha'"]
&
B' \arrow[l, shift left, "\beta'"],
\end{tikzcd}
\end{equation*}
if $\varphi\colon B\to T$ and $\varphi'\colon B'\to T$ are the
corresponding acting morphisms, where $T$ is a weak representing
object of $X$, then $\varphi=\varphi'\circ f$. This follows immediately
from the commutativity of the diagram
\[
	\vcenter{\xymatrix{
			\SplExt(B',X)  \ar[d]_-{f^*} \ar[r]^-{\tau_{B'}} & \Hom_\cC(B',T)  \ar[d]^-{-\circ f}\\
			\SplExt(B,X)  \ar[r]_-{\tau_B} & \Hom_\cC(B,T).
	}}
	\]
\end{remark}

An important observation made by G.~Janelidze in~\cite{WRA} is the following. 

\begin{theorem}\cite{WRA}\label{WRAimpliesAA}
Every weakly action representable category $\cC$ is action accessible.
\end{theorem}

To prove this result, we need two preliminary lemmas. We include their proofs, since in~\cite{WRA} they are only sketched.

\begin{lemma}\cite{WRA}\label{lemmaWRAimpliesAA}
Let $\cC$ be a weakly action representable category, and let
\begin{equation}\label{spltextlemma}
\begin{tikzcd}
X \arrow[r, "k"]
&
A \arrow[r, shift left, "\alpha"]
&
B \arrow[l, shift left, "\beta"]
\end{tikzcd}
\end{equation}
be a split extension in $\cC$ with associated acting morphism
$\varphi\colon B\to T$, where $T$ is a weak representing object of
$X$. If $\varphi$ is a monomorphism, then the split
extension~\eqref{spltextlemma} is faithful.
\end{lemma}

\begin{proof}
Let
\begin{equation}\label{splextlemma}
\begin{tikzcd}
X \arrow[r, "k'"]
&
A' \arrow[r, shift left, "\alpha'"]
&
B' \arrow[l, shift left, "\beta'"]
\end{tikzcd}
\end{equation}
be another split extension in $\cC$, and suppose that there exist two
morphisms of split extensions
\begin{equation}\label{morphismsplexts}
\begin{tikzcd}
X
    \arrow[r, "k'"]
    \arrow[d, equal]
&
A'
    \arrow[r, shift left, "\alpha'"]
    \arrow[d, shift right=.75ex, "\theta"']
    \arrow[d, shift left=.75ex, "\theta'"]
&
B'
    \arrow[l, shift left, "\beta'"]
    \arrow[d, shift right=.75ex, "f"']
    \arrow[d, shift left=.75ex, "g"]
\\
X \arrow[r, "k"]
&
A \arrow[r, shift left, "\alpha"]
&
B \arrow[l, shift left, "\beta"].
\end{tikzcd}
\end{equation}
If $\varphi'\colon B'\to T$ is the acting morphism associated with
\eqref{splextlemma}, then
\[
\varphi'=\varphi\circ f=\varphi\circ g. 
\]
Since $\varphi$ is a
monomorphism, it follows that $f=g$. Since $(k',\beta')$ is jointly
strongly epimorphic, we also obtain $\theta=\theta'$. Therefore, \eqref{spltextlemma} is a faithful object of $\SplExtX$.
\end{proof}

\begin{lemma}\cite{WRA}\label{lemma2WRAimpliesAA}
Let $\cC$ be a weakly action representable category, let $B,B'$ and~$X$ be objects of $\cC$, and let $T$ be a weak representing object of
$X$. If $f\colon B\twoheadrightarrow B'$ is a regular epimorphism,
then a morphism $\varphi'\colon B'\to T$ is an acting morphism whenever
$\varphi'\circ f$ is an acting morphism.
\end{lemma}

\begin{proof}
Let $\varphi\coloneqq\varphi'\circ f$ be an acting morphism, and let
\begin{equation}\label{splext2lemma2}
\begin{tikzcd}
X \arrow[r, "k"]
&
A \arrow[r, shift left, "\alpha"]
&
B \arrow[l, shift left, "\beta"]
\end{tikzcd}
\end{equation}
be a split extension of $B$ by $X$ whose associated acting morphism is~$\varphi$. By~\cite[Proposition~4.4]{WRA}, since $f$ is a regular epimorphism, the diagram
\[
   \vcenter{\xymatrix{
   		\SplExt(B',X)  \ar[d]_-{f^*} \ar[r]^-{\tau_{B'}} & \Hom_\cC(B',T)  \ar[d]^-{-\circ f}\\
   		\SplExt(B,X)  \ar[r]_-{\tau_B} & \Hom_\cC(B,T)
   }}
   \]
is a pullback. Therefore, there exists a split extension
\begin{equation}\label{splextlemma2}
\begin{tikzcd}
X \arrow[r, "k'"]
&
A' \arrow[r, shift left, "\alpha'"]
&
B' \arrow[l, shift left, "\beta'"]
\end{tikzcd}
\end{equation}
such that $f^*$ sends the isomorphism class of
\eqref{splextlemma2} to that of \eqref{splext2lemma2}, and
$\varphi'$ is the acting morphism associated with
\eqref{splextlemma2}.
\end{proof}

We are now ready to prove \Cref{WRAimpliesAA}.

\begin{proof}
Let $X$ be an object of $\cC$, and let
\[
\tau\colon
\SplExt(-,X)
\rightarrowtail
\Hom_{\cC}(-,T)
\]
be a weak representation of $\SplExt(-,X)$. Let
\begin{equation*}
\begin{tikzcd}
X \arrow[r, "k"]
&
A \arrow[r, shift left, "\alpha"]
&
B \arrow[l, shift left, "\beta"]
\end{tikzcd}
\end{equation*}
be a split extension in $\cC$ with associated acting morphism
$\varphi\colon B\to T$. We construct a morphism of split extensions
\begin{equation}\label{morphismsplext}
\begin{tikzcd}
X
    \arrow[r, "k"]
    \arrow[d, equal]
&
A
    \arrow[r, shift left, "\alpha"]
    \arrow[d, "\theta"']
&
B
    \arrow[l, shift left, "\beta"]
    \arrow[d, "f"]
\\
X \arrow[r, "k'"]
&
A' \arrow[r, shift left, "\alpha'"]
&
B' \arrow[l, shift left, "\beta'"]
\end{tikzcd}
\end{equation}
with faithful codomain. Let $B'=\varphi(B)$, and let
$f\colon B\twoheadrightarrow B'$ be the corestriction of $\varphi$ to
its image. Thus, $\varphi=\varphi'\circ f$, where
$\varphi'\colon B'\rightarrowtail T$ is given by the
\emph{(regular epi, mono)-factorisation} of $\varphi$.

By \Cref{lemma2WRAimpliesAA}, $\varphi'$ is an acting morphism and
therefore determines a split extension
\begin{equation}\label{splext2}
\begin{tikzcd}
X \arrow[r, "k'"]
&
A' \arrow[r, shift left, "\alpha'"]
&
B' \arrow[l, shift left, "\beta'"].
\end{tikzcd}
\end{equation}
Since $\varphi=\varphi'\circ f$, naturality and injectivity of $\tau$
show that the original split extension is isomorphic to the pullback
of~\eqref{splext2} along $f$. The pullback projection therefore gives
the required morphism $\theta$ in~\eqref{morphismsplext}.

Finally, \eqref{splext2} is faithful because its acting morphism
$\varphi'$ is a monomorphism, by \Cref{lemmaWRAimpliesAA}. Hence,
$\cC$ is action accessible.
\end{proof}

We have therefore established the implications
\[
\textit{action representability}
\Rightarrow
\textit{weak action representability}
\Rightarrow
\textit{action accessibility}.
\]

A first class of examples of categories that are action accessible but
not weakly action representable was provided by J.~R.~A.~Gray
in~\cite{Gray}. He showed that weak action representability is related
to the \emph{amalgamation property}~(AP)~\cite{kiss}.

Recall that a span of monomorphisms
$m\colon S\rightarrowtail B$ and
$m'\colon S\rightarrowtail B'$ in a category $\cC$ admits an
amalgam in $\cC$ if there exist monomorphisms
$u\colon B\rightarrowtail D$ and
$u'\colon B'\rightarrowtail D$ in $\cC$ such that
$u\circ m=u'\circ m'$. Using the (AP), J.~R.~A.~Gray obtained
sufficient conditions under which a Birkhoff subcategory of an action
representable category is not weakly action representable.

\begin{proposition}\cite{Gray}\label{Gray}
Let $\cC$ be an action representable category, and let $\mathcal{W}$
be a Birkhoff subcategory of $\cC$. Suppose that there exist
monomorphisms
$m\colon S\rightarrowtail B$ and
$m'\colon S\rightarrowtail B'$ in $\mathcal{W}$, monomorphisms
$u\colon B\rightarrowtail D$ and
$u'\colon B'\rightarrowtail D$ in $\cC$, an object~$X$ of
$\mathcal{W}$, and a monomorphism $v\colon D\rightarrowtail[X]$ such
that:
\begin{itemize}
\item[(i)] $m$ and $m'$ cannot be amalgamated in $\mathcal{W}$;
\item[(ii)] $u\circ m=u'\circ m'$;
\item[(iii)] the split extensions with kernel $X$ corresponding to
$v \circ u$ and $v \circ u'$ belong to $\mathcal{W}$.
\end{itemize}
Then $\mathcal{W}$ is not weakly action representable.
\end{proposition}

Using an example of B.~H.~Neumann consisting of an abelian group $S$,
a $2$-nilpotent group $B$, and two monomorphisms
$m,m'\colon S\rightarrowtail B$ that cannot be amalgamated in any
solvable group $D$~\cite{Neumann}, one can apply \Cref{Gray} to prove
that the categories of $n$-solvable groups are not weakly action
representable for any $n\geq3$. This result was then adapted
in~\cite{XabiMancini} to show that the categories of $k$-nilpotent
groups, for $k\geq3$, and of $2$-solvable groups are not weakly action
representable. The same article also proves that the varieties of $k$-nilpotent
Lie algebras, for $k\geq3$, and of $n$-solvable Lie algebras, for
$n\geq2$, are not weakly action representable. The situation is different when $k=2$: over a field of characteristic different from~$2$, the variety of $2$-nilpotent Lie algebras is weakly action representable
by~\cite[Theorem~2.21]{WRAAlg}, while a proof of the weak action representability of the category of $2$-nilpotent groups can be found in~\cite{DB-Mancini-Tim}.

\begin{remark}\label{remobjects}
Action representability and weak action representability are properties of the whole category~$\cC$: they require a representation, or a weak representation, of $\SplExt(-,X)$ for every object~$X$ of~$\cC$. For a single object~$X$, one may still ask whether $\SplExt(-,X)$ is representable, that is, whether $X$ admits an actor. A category which is not action representable may contain many such objects. For instance, an object $X$ of $\CAssoc$ admits an actor if and only if $\M(X)$ is commutative~\cite[Theorem~2.6]{BJK2}, and every associative algebra $X$ with $\Ann(X)=0$, or with $X^2=X$, has actor $\Bim(X)$; see the discussion following \Cref{actmorpAssoc}. More generally, the objects of an Orzech category of interest which admit an actor are characterised in terms of their universal strict general actor; see \Cref{SecUSGA}. On the other hand, to show that a category is not weakly action representable, it suffices to exhibit a single object~$X$ for which $\SplExt(-,X)$ admits no weak representation. This is the strategy underlying \Cref{Gray}.
\end{remark}

\section{Derived actions in Orzech categories of interest}\label{SecOrz}

In this section, we describe internal actions and their representability in \emph{Orzech categories of interest}~\cite{Orzech}. These are particular \emph{categories of groups with operations} in the sense of~\cite{Porter}, a notion going back to the groups with multiple operators of~\cite{Higgins}.

Let $\cC$ be a \emph{category of groups with a set of operations}
$\Omega$ and a set of identities $\mathbb{E}$, where $\mathbb{E}$
contains the group identities. If $\Omega_i$ denotes the subset of
$i$-ary operations in $\Omega$, we require the following conditions:
\begin{itemize}
\item[(i)] $\Omega=\Omega_0\cup\Omega_1\cup\Omega_2$.

\item[(ii)] The group operations, written additively as $0$, $-$, and
$+$, belong to $\Omega_0$, $\Omega_1$, and $\Omega_2$, respectively.
Set $\Omega_2'=\Omega_2\setminus\{+\}$ and
$\Omega_1'=\Omega_1\setminus\{-\}$. If $*\in\Omega_2'$, then~$\Omega_2'$ also contains the opposite operation $*^{\op}$, defined
by $x*^{\op}y\coloneqq y*x$. Moreover, $\Omega_0=\{0\}$.

\item[(iii)] For every $*\in\Omega_2'$, the set $\mathbb{E}$ contains
the identity $x*(y+z)=x*y+x*z$.

\item[(iv)] For every $\omega\in\Omega_1'$ and
$*\in\Omega_2'$, the set $\mathbb{E}$ contains the identities
$\omega(x+y)=\omega(x)+\omega(y)$ and
$\omega(x*y)=\omega(x)*y$.
\end{itemize}

For every object $X$ of $\cC$ and all $x_1,x_2,x_3\in X$, we further
require the following axioms:
\begin{itemize}
\item[(v)] \textbf{Axiom 1}: for every $*\in\Omega_2'$,
$x_1+(x_2*x_3)=(x_2*x_3)+x_1$.

\item[(vi)] \textbf{Axiom 2}: for every ordered pair
$(*,\bar{*})\in\Omega_2'\times\Omega_2'$, there exists a word $W$ in the expressions
\[
x_1(x_2x_3),\,
x_1(x_3x_2),\,
(x_2x_3)x_1,\,
(x_3x_2)x_1,
\]
\[
x_2(x_1x_3),\,
x_2(x_3x_1),\,
(x_1x_3)x_2,\,
(x_3x_1)x_2,
\]
where each juxtaposition represents an operation in $\Omega_2'$. Denoting
this word by $W(x_1;x_2,x_3;*,\bar{*})$, the identity
\[
(x_1*x_2)\mathbin{\bar{*}}x_3
=
W(x_1;x_2,x_3;*,\bar{*})
\]
belongs to $\mathbb{E}$.
\end{itemize}

\begin{definition}\cite{Orzech}\label{definterest}
A category of groups with operations satisfying conditions~(i)--(vi)
is called an \emph{Orzech category of interest}, or simply a
\emph{category of interest}.
\end{definition}

Let $\mathbb{E}'$ be the subset of $\mathbb{E}$ containing the group identities and the identities in~(iii) and~(iv) (and, for categories of $\bF$-algebras, also the identities expressing that $+$ is commutative and that the operations $\omega_\lambda$ define a structure of $\bF$-vector space), and let $\cC'$ be the
corresponding category of groups with operations. We therefore obtain
a forgetful functor $U\colon\cC\to\cC'$. The category $\cC'$ is called
the \emph{general category of groups with operations} associated with
the Orzech category of interest $\cC$. Here and in what follows, for every $*\in\Omega_2'$ we regard $*^{\op}$ as a derived operation, that is, we also include in $\mathbb{E}'$ the identity $x*^{\op}y=y*x$. In particular, if $\cC$ is a category of $\bF$-algebras with a single bilinear operation~$*$, so that $\Omega_2'=\{*,*^{\op}\}$, then $\cC'=\Alg$.

\begin{remark}
If $\cC$ is a category of abelian groups with operations, that is, if
the group operation $+$ is commutative, then Axiom~1 is automatically
satisfied.
\end{remark}

We now present some examples and counterexamples of Orzech categories
of interest.

\begin{example}
The categories $\Grp$ of groups and $\Ab$ of abelian groups are Orzech
categories of interest, with
$\Omega_1'=\Omega_2'=\emptyset$.
\end{example}

\begin{example}
The categories $\Rng$ of not necessarily unital rings and $\CRng$ of
commutative not necessarily unital rings are Orzech categories of
interest, with $\Omega_1'=\emptyset$ and
$\Omega_2'=\{*,*^{\op}\}$, where $*$ denotes ring multiplication.
Axiom~2 is given by the associativity of $*$.
\end{example}

\begin{example}
Let $\cC=\Alg$ be the category of \emph{all} non-associative algebras
over a field $\bF$. Its objects are $\bF$-vector spaces $X$ endowed
with a bilinear operation
\[
X\times X\to X\colon(x,y)\mapsto x*y,
\]
and its morphisms are the $\bF$-linear maps preserving this operation.
The category $\Alg$ is a category of abelian groups with operations,
with $\Omega_1'=\{\omega_\lambda\mid\lambda\in\bF\}$, where
$\omega_\lambda(x)=\lambda x$, and $\Omega_2'=\{*,*^{\op}\}$, but it is not an Orzech category of
interest because Axiom~2 fails.
\end{example}

\begin{example}\label{exalgint}
The category $\Assoc$ of associative algebras over a field $\bF$ is an
Orzech category of interest. Its underlying group operation is
commutative, and
\[
\Omega_1'=\{\omega_\lambda\mid\lambda\in\bF\},
\qquad
\Omega_2'=\left\{\ast,\ast^{\op}\right\},
\]
where $\ast$ denotes the associative multiplication. Axiom~2 is given
by associativity:
\[
(x_1\ast x_2)\ast x_3=x_1\ast(x_2\ast x_3).
\]
Similarly, the category $\Lie$ of Lie algebras over $\bF$ is an Orzech
category of interest, with
\[
\Omega_1'=\{\omega_\lambda\mid\lambda\in\bF\},
\qquad
\Omega_2'=
\left\{[-,-],[-,-]^{\op}\right\},
\]
where $[-,-]$ denotes the Lie bracket, and Axiom~2 is given by the
Jacobi identity. Moreover, $\Assoc'=\Lie'=\Alg$.
\end{example}

\begin{example}\label{esLeib}
A \emph{(right) Leibniz algebra}~\cite{loday} over a field $\bF$ is a
non-associative $\bF$-algebra $X$ endowed with a bilinear operation
$[-,-]$ satisfying the \emph{(right) Leibniz identity}
\[
[[x,y],z]=[[x,z],y]+[x,[y,z]]
\]
for all $x,y,z\in X$. Lie algebras are precisely the Leibniz algebras with an alternating bracket, that is, $[x,x]=0$ for every $x \in X$, since in this case the Leibniz identity is equivalent to the Jacobi identity.

The category $\Leib$ of right Leibniz algebras over $\bF$ is an Orzech
category of interest, with
\[
\Omega_1'=\{\omega_\lambda\mid\lambda\in\bF\},
\qquad
\Omega_2'=
\left\{[-,-],[-,-]^{\op}\right\}.
\]
Axiom~2 may be expressed through the rewriting identities
\begin{align*}
[x,[y,z]]
  &= [[x,y],z]-[[x,z],y],\\
[[y,z],x]
  &= [[y,x],z]-[y,[x,z]],
\end{align*}
the first of which is equivalent to the Leibniz identity, while the second one follows from it, since $[y,[x,z]+[z,x]]=0$ in every right Leibniz algebra. Moreover, $\Leib'=\Alg$.
\end{example}

\begin{example}
Recall that a \emph{Jordan algebra} over a field $\bF$ is a
commutative, not necessarily associative, $\bF$-algebra $(X,\cdot)$
satisfying the \emph{Jordan identity}
\[
(xy)(xx)=x(y(xx))
\]
for all $x,y\in X$.

The category $\Jord$ of Jordan algebras over $\bF$ is not an Orzech
category of interest because Axiom~2 is not satisfied.
\end{example}

Every category of groups with operations is semi-abelian; see, for instance,~\cite{BB,BJ03}. Hence, split extensions are naturally in bijection with internal actions. In
this context, however, it is more convenient to describe internal
actions in terms of \emph{derived actions}.

\begin{definition}\label{def derived actions}
Let $\cC$ be a category of groups with operations, and let
\begin{equation}\label{diag:SplC}
\begin{tikzcd}
X \arrow[r, "k"]
&
A \arrow[r, shift left, "\alpha"]
&
B \arrow[l, shift left, "\beta"]
\end{tikzcd}
\end{equation}
be a split extension in $\cC$. The collection of maps
\[
\{f_*\colon B\times X\to X\mid *\in\Omega_2\}
\]
defined by
\begin{equation*}
b\cdot x
=
k^{-1}\bigl(\beta(b)+k(x)-\beta(b)\bigr),
\qquad
b*x
=
k^{-1}\bigl(\beta(b)*k(x)\bigr),
\end{equation*}
where
\[
b\cdot x\coloneqq f_+(b,x),
\quad
b*x\coloneqq f_*(b,x)
\quad
\text{for every }*\in\Omega_2',
\]
is called the \emph{derived action} of $B$ on $X$ associated
with~\eqref{diag:SplC}. These expressions are well defined, since $\beta(b)+k(x)-\beta(b)$ and $\beta(b)*k(x)$ belong to
$\ker(\alpha)$, and the corestriction
$k\colon X\to\ker(\alpha)$ is an isomorphism.
\end{definition}

Given two objects $B$ and $X$ of $\cC$ and a collection of maps
$\{f_*\}_{*\in\Omega_2}$, denoted by
\[
b\cdot x\coloneqq f_+(b,x),
\qquad
b*x\coloneqq f_*(b,x)
\quad
\text{for every }*\in\Omega_2',
\]
we may define binary operations on the Cartesian product $B\times X$ by
\[
(b,x)+(b',x')=(b+b',x+b\cdot x')
\]
and
\[
(b,x)*(b',x')
=
(b*b',x*x'+b*x'+x*b'),
\]
where $x*b'\coloneqq b'*^{\op}x$. We also equip $B\times X$ with the constant and unary operations
\[
0=(0,0),\qquad
-(b,x)=(-b,(-b)\cdot(-x)),\qquad
\omega(b,x)=(\omega(b),\omega(x))
\]
for every $\omega\in\Omega_1'$.

\begin{definition}
The algebra on $B\times X$ defined by these operations is called the \emph{semidirect product} of $B$ and $X$ with respect to $\{f_*\}_{*\in\Omega_2}$, and it is denoted by $B\ltimes X$.
\end{definition}

The collection $\{f_*\}_{*\in\Omega_2}$ defines a derived action of
$B$ on $X$ if and only if $B\ltimes X$ is an object of $\cC$. In this
case, one obtains the split extension in $\cC$
\begin{equation}\label{splint}
\begin{tikzcd}
X \arrow[r, "\iota_2"]
&
B\ltimes X \arrow[r, shift left, "\pi_1"]
&
B \arrow[l, shift left, "\iota_1"],
\end{tikzcd}
\end{equation}
where $\iota_2(x)=(0,x)$, $\iota_1(b)=(b,0)$, and
$\pi_1(b,x)=b$. Moreover, $\{f_*\}_{*\in\Omega_2}$ is precisely the
derived action associated with~\eqref{splint}.

\begin{remark}\label{remabgrps}
When the group operation $+$ is commutative, one has $b\cdot x=x$ for
all~$b\in B$ and $x\in X$. Hence, a derived action of $B$ on $X$ is
determined by a collection of maps
$\{f_*\}_{*\in\Omega_2'}$.
\end{remark}

\begin{example}
Let $B$ and $X$ be objects of $\Alg$. A pair of maps
\[
f_*\colon B\times X\to X,
\qquad
f_{*^{\op}}\colon B\times X\to X
\]
defines a derived action of $B$ on $X$ if and only if the vector space
$B\oplus X\cong B\times X$, endowed with the binary operation
\[
(b,x)*(b',x')
=
(b*b',x*x'+b*x'+x*b'),
\]
is a non-associative algebra over $\bF$. This happens if and only if
$f_*$ and $f_{*^{\op}}$ are bilinear maps.
\end{example}

When $\cC$ is an Orzech category of interest, a collection
$\{f_*\}_{*\in\Omega_2}$ defines a derived action of $B$ on $X$ if and only if its maps satisfy the conditions described below.

\begin{proposition}\cite{Casas, Datuashvili}\label{conditionactions}
Let $\cC$ be an Orzech category of interest, and let $B$ and $X$ be
objects of $\cC$. Consider a collection of maps
$\{f_*\colon B\times X\to X\mid *\in\Omega_2\}$, denoted by
\[
b\cdot x\coloneqq f_+(b,x),
\qquad
b*x\coloneqq f_*(b,x)
\quad
\text{for every }*\in\Omega_2'.
\]
This collection defines a derived action of $B$ on $X$ in $\cC$ if
and only if the following conditions are satisfied:
\begin{itemize}
\item[(1)] $0\cdot x=x$;
\item[(2)] $b\cdot(x+x')=b\cdot x+b\cdot x'$;
\item[(3)] $(b+b')\cdot x=b\cdot(b'\cdot x)$;
\item[(4)] $b*(x+x')=b*x+b*x'$;
\item[(5)] $(b+b')*x=b*x+b'*x$;
\item[(6)] $b\cdot(x*x')=x*x'$;
\item[(7)] $b\cdot(x*b')=x*b'$;
\item[(8)] $(b*b')\cdot x=x$;
\item[(9)] $x*(b\cdot x')=x*x'$;
\item[(10)] $b*(b'\cdot x)=b*x$;
\item[(11)] $\omega(b\cdot x)=\omega(b)\cdot\omega(x)$;
\item[(12)] $\omega(x*b)=\omega(x)*b=x*\omega(b)$;
\item[(13)] $\alpha+\beta*\gamma=\beta*\gamma+\alpha$;
\item[(14)] for every ordered pair
$(*,\bar{*})\in\Omega_2'\times\Omega_2'$, the word $W$ fixed in Axiom~2 satisfies
\[
(x_1*x_2)\mathbin{\bar{*}}x_3
=
W(x_1;x_2,x_3;*,\bar{*});
\]
\item[(15)] every remaining identity in the fixed defining
presentation of $\cC$ holds in $B\ltimes X$, evaluated on
arbitrary elements of $B\times X$.
\end{itemize}
These conditions hold for all $\omega\in\Omega_1'$,
$*\in\Omega_2'$, $b,b'\in B$, $x,x'\in X$, and
$\alpha,\beta,\gamma\in B\cup X$ for which the expressions involved are defined, as well as for all $x_1,x_2,x_3\in B\cup X$ for which both sides of~(14) are defined.
\end{proposition}

\begin{corollary}\label{conditionsactionsabgroups}
Let $\cC$ be an Orzech category of interest whose objects are abelian
groups with operations, and let $B$ and $X$ be objects of $\cC$. A
collection of maps
$\{f_*\colon B\times X\to X\mid *\in\Omega_2'\}$, together with the
trivial group action $b\cdot x=x$, defines a derived
action of $B$ on $X$ in $\cC$ if and only if it satisfies
conditions~(4), (5), (12), (14), and~(15) of
\Cref{conditionactions}.
\end{corollary}

\begin{example}
Let $\cC=\Lie$. Since the group operation $+$ is commutative, a
derived action of $B$ on $X$ is given by a pair of bilinear maps
\[
f_{[-,-]}\colon B\times X\to X,
\qquad
f_{[-,-]^{\op}}\colon B\times X\to X
\]
such that the bilinear operation
\[
[(b,x),(b',x')]
=
([b,b'],[x,x']+[b,x']+[x,b'])
\]
defines a Lie algebra structure on the vector space
$B\oplus X$, where
$[b,x]\coloneqq f_{[-,-]}(b,x)$ and
$[x,b]\coloneqq f_{[-,-]^{\op}}(b,x)$. By
\Cref{conditionsactionsabgroups}, this happens if and only if
\begin{itemize}
\item[(1)] $[x,b]=-[b,x]$;
\item[(2)] $[b,[x,x']]=[[b,x],x']+[x,[b,x']]$;
\item[(3)] $[[b,b'],x]=[b,[b',x]]-[b',[b,x]]$;
\end{itemize}
for all $b,b'\in B$ and $x,x'\in X$.

For instance, let $B=\Der(X)$ be the Lie algebra of derivations of
$X$. The canonical action of $\Der(X)$ on $X$, defined by
\[
[d,x]=-[x,d]\coloneqq d(x),
\]
for all $d\in\Der(X)$ and $x\in X$, is a derived action in $\Lie$,
since the semidirect product
$\Hol(X)=\Der(X)\ltimes X$, the holomorph of $X$, is a
Lie algebra.

By contrast, the canonical action of $\Bim(X)$ on $X$ in $\Assoc$,
defined by
\[
(f\ast-,-\ast f)\ast x\coloneqq f\ast x
\]
and
\[
x\ast(f\ast-,-\ast f)\coloneqq x\ast f,
\]
for all $(f\ast-,-\ast f)\in\Bim(X)$ and $x\in X$, is not, in
general, a derived action. Indeed, a derived action of $B$ on $X$ in
$\Assoc$ is given by a pair of bilinear maps
\[
f_*\colon B\times X\to X,
\qquad
f_{*^{\op}}\colon B\times X\to X,
\]
with $b*x\coloneqq f_*(b,x)$ and
$x*b\coloneqq f_{*^{\op}}(b,x)$, satisfying
\begin{itemize}
\item[(1)] $(b*b')*x=b*(b'*x)$;
\item[(2)] $x*(b*b')=(x*b)*b'$;
\item[(3)] $(b*x)*b'=b*(x*b')$;
\item[(4)] $b*(x*x')=(b*x)*x'$;
\item[(5)] $(x*x')*b=x*(x'*b)$;
\item[(6)] $x*(b*x')=(x*b)*x'$;
\end{itemize}
for all $b,b'\in B$ and $x,x'\in X$.

As observed in \Cref{exnotacting}, given
$(f\ast-,-\ast f),(f'\ast-,-\ast f')\in\Bim(X)$, condition~(3) need
not be satisfied. For instance, let $X=\bF\{e_1,e_2\}$ be the
two-dimensional abelian algebra, and let
$(f\ast-,-\ast f)\in\Bim(X)
=\End(X)\times\End(X)^{\op}$ be defined by
\[
f\ast e_1=e_2,
\quad
f\ast e_2=0,
\qquad
e_1\ast f=e_1,
\quad
e_2\ast f=e_1.
\]
Then
\[
(f\ast e_1)\ast f
=
e_2\ast f
=
e_1
\neq
e_2
=
f\ast e_1
=
f\ast(e_1\ast f).
\]
Thus, condition~(3) is not satisfied, and the semidirect product
$\Bim(X)\ltimes X$ is not an associative algebra. This gives an
example of a derived action in $\Assoc'=\Alg$ that is not a derived
action in $\Assoc$.
\end{example}

\begin{remark}\label{reminterest}
In summary, derived actions can be defined in any category of groups
with operations and can be used to describe split extensions, or
equivalently internal actions.

Starting with a split extension in $\cC$
\begin{equation}\label{splrem}
\begin{tikzcd}
X \arrow[r, "k"]
&
A \arrow[r, shift left, "\alpha"]
&
B \arrow[l, shift left, "\beta"],
\end{tikzcd}
\end{equation}
we obtain the derived action
$\{f_*\}_{*\in\Omega_2}$ of $B$ on $X$ given by
\begin{equation*}
b\cdot x
=
k^{-1}\bigl(\beta(b)+k(x)-\beta(b)\bigr),
\qquad
b*x
=
k^{-1}\bigl(\beta(b)*k(x)\bigr),
\end{equation*}
for all $b\in B$, $x\in X$, and $*\in\Omega_2'$. Conversely, given a
collection of maps $\{f_*\}_{*\in\Omega_2}$, denoted by
\[
b\cdot x\coloneqq f_+(b,x),
\qquad
b*x\coloneqq f_*(b,x)
\quad
\text{for every }*\in\Omega_2',
\]
one can construct the semidirect product $B\ltimes X$ as above. The
collection $\{f_*\}_{*\in\Omega_2}$ is a derived action of $B$ on $X$
if and only if $B\ltimes X$ is an object of $\cC$. In this case, it is
the derived action associated with the split extension
\begin{equation}\label{splrem2}
\begin{tikzcd}
X \arrow[r, "\iota_2"]
&
B\ltimes X \arrow[r, shift left, "\pi_1"]
&
B \arrow[l, shift left, "\iota_1"],
\end{tikzcd}
\end{equation}
where $\iota_2(x)=(0,x)$, $\iota_1(b)=(b,0)$, and $\pi_1(b,x)=b$. Moreover, the diagram
\begin{equation*}
\begin{tikzcd}
X
    \arrow[r, "\iota_2"]
    \arrow[d, equal]
&
B\ltimes X
    \arrow[r, shift left, "\pi_1"]
    \arrow[d, "\theta"']
&
B
    \arrow[l, shift left, "\iota_1"]
    \arrow[d, equal]
\\
X \arrow[r, "k"]
&
A \arrow[r, shift left, "\alpha"]
&
B \arrow[l, shift left, "\beta"],
\end{tikzcd}
\end{equation*}
where $\theta(b,x)=k(x)+\beta(b)$, defines an isomorphism of split extensions between~\eqref{splrem2} and~\eqref{splrem}.

We therefore obtain a bijection
\[
\eta_B\colon\SplExt(B,X)\cong\DAct(B,X)
\]
between the set $\SplExt(B,X)$ of isomorphism classes of split
extensions of $B$ by $X$ in $\cC$ and the set $\DAct(B,X)$ of derived
actions of $B$ on $X$ in $\cC$.
\end{remark}

Denote by $\DAct(-,X)$ the functor $\cC^{\op}\to\Set$ assigning to
each object $B$ of $\cC$ the set $\DAct(B,X)$. Given a morphism
$g\colon B'\to B$, the map $\DAct(g,X)$ sends a derived action
$\{f_*\}_{*\in\Omega_2}$ of $B$ on $X$ to the derived action
$\{f'_*\}_{*\in\Omega_2}$ of $B'$ on $X$, where
$f'_*\coloneqq f_*\circ(g\times \id_X)$, that is,
\begin{equation}\label{eqf'}
f'_+(b',x)=f_+(g(b'),x),
\qquad
f'_*(b',x)=f_*(g(b'),x),
\end{equation}
for all $b'\in B'$, $x\in X$, and $*\in\Omega_2'$. We show that if
$\{f_*\}_{*\in\Omega_2}$ is the derived action associated with a split
extension
\begin{equation}\label{splBint}
\begin{tikzcd}
X \arrow[r, "k"]
&
A \arrow[r, shift left, "\alpha"]
&
B \arrow[l, shift left, "\beta"],
\end{tikzcd}
\end{equation}
then $\{f'_*\}_{*\in\Omega_2}$ is the derived action associated with
the split extension
\begin{equation}\label{splB'int}
\begin{tikzcd}
X \arrow[r, "k'"]
&
A' \arrow[r, shift left, "\alpha'"]
&
B' \arrow[l, shift left, "\beta'"]
\end{tikzcd}
\end{equation}
obtained by applying the change of base function
$g^*=\SplExt(g,X)$. In other words, we have the following result.

\begin{theorem}\label{intderivedact}
The bijections $\eta_B$ extend to a natural isomorphism of functors
\[
\eta\colon\SplExt(-,X)\cong\DAct(-,X).
\]
\end{theorem}

\begin{proof}
Let $g\colon B'\to B$ be a morphism in $\cC$, and suppose that:
\begin{itemize}
\item[(1)] $g^*=\SplExt(g,X)$ sends the isomorphism class of
\eqref{splBint} to that of \eqref{splB'int};
\item[(2)] $\DAct(g,X)$ sends the derived action
$\{f_*\}_{*\in\Omega_2}$ associated with \eqref{splBint} to the
collection $\{f'_*\}_{*\in\Omega_2}$ defined
in~\eqref{eqf'}.
\end{itemize}
Let $\theta\colon A'\to A$ be the morphism obtained by pulling back
along $g$. Then
\[
\theta\circ\beta'=\beta\circ g \quad \text{ and } \quad \theta\circ k'=k.
\]
We must show that $\eta_B$ is natural in $B$, that is, that the
following diagram in $\Set$ commutes:
\[
\vcenter{\xymatrix{
		\SplExt(B,X)  \ar[d]_-{g^*} \ar[r]^-{\eta_{B}} & \DAct(B,X)  \ar[d]^-{\DAct(g,X)}\\
		\SplExt(B',X)  \ar[r]_-{\eta_{B'}} & \DAct(B',X).
}}
\]
The composition $\eta_{B'}\circ g^*$ sends~\eqref{splBint} to the derived action of $B'$ on $X$ associated with
\eqref{splB'int}. Denote this derived action by
$\{F'_*\}_{*\in\Omega_2}$, so that
\[
k'\bigl(F'_+(b',x)\bigr)
=
\beta'(b')+k'(x)-\beta'(b'),
\qquad
k'\bigl(F'_*(b',x)\bigr)
=
\beta'(b')*k'(x),
\]
for all $b'\in B'$, $x\in X$, and $*\in\Omega_2'$. Applying $\theta$
and using the pullback identities, we obtain
\begin{align*}
k\bigl(F'_+(b',x)\bigr)
&=\theta\bigl(k'(F'_+(b',x))\bigr)\\
&=\theta\bigl(\beta'(b')+k'(x)-\beta'(b')\bigr)\\
&=\beta(g(b'))+k(x)-\beta(g(b'))\\
&=k\bigl(f_+(g(b'),x)\bigr)\\
&=k\bigl(f'_+(b',x)\bigr)
\end{align*}
and
\begin{align*}
k\bigl(F'_*(b',x)\bigr)
&=\theta\bigl(k'(F'_*(b',x))\bigr)\\
&=\theta\bigl(\beta'(b')*k'(x)\bigr)\\
&=\beta(g(b'))*k(x)\\
&=k\bigl(f_*(g(b'),x)\bigr)\\
&=k\bigl(f'_*(b',x)\bigr).
\end{align*}
Since $k$ is injective, it follows that
\[
F'_+(b',x)=f_+(g(b'),x)=f'_+(b',x),
\qquad
F'_*(b',x)=f_*(g(b'),x)=f'_*(b',x)
\]
for every $*\in\Omega_2'$. Therefore,
$\eta_{B'}\circ g^*=\DAct(g,X)\circ\eta_B$, proving that
\[
\eta\colon\SplExt(-,X)\cong\DAct(-,X)
\]
is a natural isomorphism.
\end{proof}

Finally, since every category of groups with operations $\cC$ is
semi-abelian, by \Cref{actspltnatural} we have a natural isomorphism
\[
\Act(-,X)\cong\SplExt(-,X).
\]
Composing this isomorphism with the one of \Cref{intderivedact}, we obtain
\[
\Act(-,X)\cong\DAct(-,X).
\]
Thus, internal actions are in natural bijection with derived
actions.

\subsection{The universal strict general actor}\label{SecUSGA}

In~\cite{Casas}, the authors studied the representability of actions
in Orzech categories of interest. They proved that, for every object
$X$ of an Orzech category of interest $\cC$, there exists an object
$\USGA(X)$ of~$\cC'$, called the \emph{universal strict general actor}
of $X$, with the following properties:
\begin{itemize}
\item[(1)] The object $\USGA(X)$ is a \emph{general actor} of $X$.
This means that $\USGA(X)$ has a derived action on $X$ in $\cC'$, and
for every object $B$ of $\cC$ and every derived action
$\{f_*\}_{*\in\Omega_2}$ of $B$ on $X$ in $\cC$, there exists a unique
morphism $\varphi\colon B\to\USGA(X)$ in $\cC'$ such that the given
derived action is recovered from the action of $\varphi(B)$ on $X$:
\[
b\cdot x=\varphi(b)\cdot x,
\qquad
b*x=\varphi(b)*x,
\]
for all $b\in B$, $x\in X$, and $*\in\Omega_2'$.

\item[(2)] The object $\USGA(X)$ is a \emph{strict general actor}; that is, the derived action of $\USGA(X)$ on $X$ is
\emph{strict}\footnote{Recall from~\cite{Casas} that a derived action
of $B$ on $X$ is said to be \emph{strict} if, for $b,b'\in B$, the
conditions $b\cdot x=b'\cdot x$,
$\omega(b)\cdot x=\omega(b')\cdot x$, and $b*x=b'*x$, for all
$x\in X$, $\omega\in\Omega_1'$, and $*\in\Omega_2'$, imply that
$b=b'$.}.

\item[(3)] Let $\{B_j\}_{j\in J}$ be a family representing all objects
of $\cC$ equipped with a derived action on $X$ in $\cC$, and let
$\varphi_j\colon B_j\to\USGA(X)$ be the corresponding morphisms
provided by~(1). The elements of $\USGA(X)$ satisfy
\[
\bigl(\varphi_i(b)*\varphi_j(b')\bigr)\mathbin{\bar{*}}x
=
W\bigl(\varphi_i(b);\varphi_j(b'),x;*,\bar{*}\bigr),
\]
for all $b\in B_i$, $b'\in B_j$,  $x \in X$,
$*,\bar{*}\in\Omega_2'$, and $i,j\in J$ for which the expressions are defined.

\item[(4)] The object $\USGA(X)$ satisfies the following universal
property: for every strict general actor $\operatorname{SGA}(X)$
satisfying~(3), there exists a unique morphism
$\eta\colon\USGA(X)\to\operatorname{SGA}(X)$ in $\cC'$ such that
$\eta\circ\varphi_j=\psi_j$ for every $j\in J$, where
$\varphi_j\colon B_j\to\USGA(X)$ and
$\psi_j\colon B_j\to\operatorname{SGA}(X)$ are the morphisms
corresponding to the derived actions of the objects $B_j$.
\end{itemize}

The results of~\cite{Casas} also show that Orzech categories of
interest need not be action representable. More precisely, an object
$X$ of an Orzech category of interest $\cC$ admits an actor precisely
when its universal strict general actor induces a derived action in
$\cC$.

\begin{theorem}\cite{Casas}
Let $\cC$ be an Orzech category of interest, and let $X$ be an object
of $\cC$. Then $X$ admits an actor $[X]$ if and only if the semidirect
product $\USGA(X)\ltimes X$ is an object of $\cC$. In this case,
$[X]\cong\USGA(X)$.
\end{theorem}

Later, A.~Montoli proved in~\cite[Theorem~4.7]{Montoli} that
every Orzech category of interest is action accessible. However, as
recalled in \Cref{SecWRA}, J.~R.~A.~Gray observed that such a category need not be weakly action representable. Nevertheless, the results of~\cite{Casas} imply the following.

\begin{proposition}\cite{CigoliManciniMetere}\label{USGA}
Let $\cC$ be an Orzech category of interest, and let $X$ be an object
of $\cC$. There exists a natural monomorphism of functors
\[
\mu\colon
\SplExt(-,X)
\rightarrowtail
\Hom_{\cC'}\bigl(U(-),\USGA(X)\bigr),
\]
where $\SplExt(-,X)=\SplExt_{\cC}(-,X)$ and
$U\colon\cC\to\cC'$ is the forgetful functor. If, moreover,
$\USGA(X)$ is an object of $\cC$, then $(\USGA(X),\mu)$ is a weak
representation of $\SplExt(-,X)$.
\end{proposition}

\begin{corollary}\cite{CigoliManciniMetere}\label{USGA2}
Let $\cC$ be an Orzech category of interest. If $\USGA(X)$ is an
object of $\cC$ for every object $X$ of $\cC$, then $\cC$ is weakly action
representable.
\end{corollary}

\begin{example}\label{USGA_Assoc}
Let $\cC=\Assoc$, with the presentation of Axiom~2 given in \Cref{exalgint}, and let $X$ be an associative algebra. By~\cite[Proposition~5.6]{Casas}, the universal strict general actor $\USGA(X)$ is an associative algebra. Hence, by \Cref{USGA2}, the category $\Assoc$ is 
weakly action representable~\cite{WRA}.
\end{example}

\begin{example}\label{USGA_Leib}
Let $\cC=\Leib$, and let $X$ be a Leibniz algebra. The presentation of Axiom~2 adopted in \Cref{esLeib} is the one of~\cite[Definition~5.1]{Casas}. With this presentation, the universal strict general actor $\USGA(X)$ is a Leibniz algebra; see~\cite[Proposition~5.4]{Casas}. Hence, by \Cref{USGA2}, the category $\Leib$ is weakly action representable. For the other presentation considered in~\cite[Definition~5.2]{Casas}, \cite[Proposition~5.4]{Casas} shows that $\USGA(X)$ is a Leibniz algebra whenever $X$ satisfies~\cite[Condition~1]{Casas}, for instance when $\Ann(X)=0$ or $[X,X]=X$; see also \Cref{exLeib}. We refer the reader to~\cite{CigoliManciniMetere} for an explicit description of a weak
representation and of the corresponding acting morphisms in $\Leib$.
\end{example}

\begin{remark}\label{presUSGA}
Condition~(3) in the construction of $\USGA$ is expressed in terms of
the identities appearing in Axiom~2. An Orzech category of interest
$\cC$ may admit several equivalent presentations, obtained by
choosing different but equivalent systems of identities for Axiom~2.
The construction of the universal strict general actor may therefore
depend on the chosen presentation: equivalent presentations of the
same category may give rise to non-isomorphic universal strict general
actors.
\end{remark}

\begin{remark}
Universal strict general actors and actors have been studied in several other categories of interest and related settings. For instance, actors of crossed modules of groups, of crossed modules of Lie algebras, of crossed modules of Leibniz algebras and of precrossed modules were constructed in~\cite{Norrie,CasasLadra,CFGK,CDLprecrossed}, respectively, while actors of Lie--Leibniz algebras and universal strict general actors of alternative algebras were studied in~\cite{CDLLieLeib,CDLAlt}.
\end{remark}

\section{Varieties of non-associative algebras}\label{SecVar}

Our main algebraic setting is that of varieties of non-associative algebras over a fixed field $\bF$. By such a variety, we mean a class of non-associative $\bF$-algebras determined by a fixed set of polynomial identities.

\begin{definition}\cite{VdL-NAA}
A \emph{non-associative algebra} $X$ is a vector space over a field
$\bF$ endowed with a bilinear operation
$X\times X\to X\colon(x,y)\mapsto xy$, called the
\emph{multiplication} or \emph{product}.
\end{definition}

\begin{remark}
In some cases, such as Lie and Leibniz algebras, we refer to the
multiplication as the \emph{commutator} or \emph{bracket} and denote
it by $[x,y]$.
\end{remark}

\begin{remark}
We do not assume the existence of a unit or impose any other condition
on the multiplication besides bilinearity.
\end{remark}

Consider the free-algebra functor
\[
\bF[\bFM(-)]\colon\Set\to\Alg,
\]
which sends a set $S$ to the free non-associative $\bF$-algebra
generated by $S$. This functor is left adjoint to the forgetful functor
$\Alg\to\Set$ and factors as
\[
\Set
\xrightarrow{\;\bFM\;}
\Mag
\xrightarrow{\;\bF[-]\;}
\Alg.
\]
Here, $\Mag$ denotes the category of \emph{magmas}, $\bFM$ is the
free-magma functor, so that $\bFM(S)$ is the magma of non-associative
words in the elements of $S$, and $\bF[-]$ is the magma-algebra
functor.

\begin{definition}
Let $S$ be a set. An element $\varphi\in\bF[\bFM(S)]$ is called a
\emph{non-associative polynomial} in the variables $S$. A non-zero
polynomial is called a \emph{monomial} if it is a scalar multiple of
an element of $\bFM(S)$.
\end{definition}

For example, if $S=\{x,y,z,t\}$, then
\[
(xy)t+(zy)x,\qquad xx+yz,\qquad (xt)(yz)
\]
are non-associative polynomials, but only the last one is a monomial.

Let $\varphi$ be a non-zero monomial in the variables
$\{x_1,\dots,x_n\}$. Its \emph{type} is the $n$-tuple
\[
(k_1,\dots,k_n)\in\N^n,
\]
where $k_i$ is the number of occurrences of $x_i$ in $\varphi$, and
its \emph{degree} is
\[
k_1+\dots+k_n.
\]
The \emph{degree} of a non-zero polynomial is the maximum of the degrees of its monomials.

A polynomial is said to be:
\begin{itemize}
\item[(1)] \emph{homogeneous} if all its monomials have the same type;
\item[(2)] \emph{multilinear} if all its monomials have type $(1,\dots,1)$.
\end{itemize}

We observe that a multilinear polynomial in the variables $x_1,\dots,x_n$ has degree~$n$.

For instance, if $S=\{x,y,z\}$, then the monomial
\[
(xx)(yz)
\]
has type $(2,1,1)$ and degree $4$. The polynomial
\[
(xx)y+x(xy)
\]
is homogeneous of type $(2,1,0)$ and degree $3$, but it is not
multilinear. On the other hand,
\[
(xy)z-x(yz)
\]
is a multilinear polynomial, since both of its monomials have type
$(1,1,1)$. Notice also that
\[
(xx)y+xyz
\]
is not well defined without specifying the parentheses in the second monomial; for example, one must write either $(xy)z$ or $x(yz)$.
Finally,
\[
(xx)y+(yz)x
\]
is not homogeneous, since its monomials have types $(2,1,0)$ and
$(1,1,1)$, respectively.

\begin{remark}
Since our algebras need not have units, the polynomials considered
here have no constant terms.
\end{remark}

\begin{definition}\label{def identity variety}
A \emph{(polynomial) identity} of an algebra $X$ is a non-associative
polynomial $\varphi=\varphi(x_1,\dots,x_n)$ such that
\[
\varphi(a_1,\dots,a_n)=0
\]
for all $a_1,\dots,a_n\in X$. In this case, we say that $X$
satisfies the identity $\varphi$.
\end{definition}

\begin{definition}
Let $S$ be a set of variables and let
$I\subseteq\bF[\bFM(S)]$. The \emph{variety of non-associative
algebras determined by $I$} is the class of all non-associative
$\bF$-algebras satisfying every identity in $I$. More generally, we
say that a variety satisfies the identities in $I$ if every algebra
of the variety satisfies them.
\end{definition}

\begin{definition}
Let $\V$ be a variety of non-associative algebras. We say that:
\begin{itemize}
\item[(1)] $\V$ is \emph{homogeneous} if, for every identity $\varphi=0$ of~$\V$, each \emph{homogeneous component} of~$\varphi$, that is, the sum of the monomials of~$\varphi$ of a given type, gives rise to an identity of~$\V$;
\item[(2)] $\V$ is \emph{operadic}~\cite{Cosmash} if it is determined by a set of multilinear identities;
\item[(3)] $\V$ is \emph{quadratic} if there exists a set of identities of degrees~$2$ and~$3$ generating all the identities of $\V$\footnote{Recall that an operadic, quadratic variety of algebras over a field can be viewed as a variety determined by a quadratic \emph{operad}; see~\cite{DGM}.}.
\end{itemize}
\end{definition}

Every variety of non-associative algebras $\V$ can be seen as a
full subcategory of $\Alg$, with the same morphisms, and there is a
forgetful functor $\V\to\Alg$. Moreover, every such variety is a semi-abelian category, since it is a category of groups with operations in the sense of \Cref{SecOrz}.

\begin{remark}
If the field $\bF$ is infinite, then every variety $\V$ of non-associative $\bF$-algebras is homogeneous~\cite{hmg}. Moreover,
if $\bchar(\bF)=0$, then every such variety is operadic. This follows from the multilinearisation process; see
\cite[Corollary~3.7]{Osborn}. Finally, over an arbitrary field, every operadic variety is homogeneous, since substituting polynomials into a multilinear identity yields a linear combination of homogeneous identities.
\end{remark}

\begin{remark}
When considering a variety $\V$ of non-associative algebras, we assume that the field $\bF$ is fixed and omit it from the notation.
\end{remark}

\begin{examples}\label{Examples varieties}{\ }
\begin{enumerate}
\item The variety $\AbAlg$ of \emph{abelian algebras} is determined
by the identity $xy=0$. Recall that $\AbAlg$ is an abelian category
and is therefore action representable.

\item $\Com$ and $\ACom$ denote, respectively, the
varieties of commutative and anti-commutative, not necessarily
associative, algebras. These varieties are determined, respectively,
by the identities $xy-yx=0$ and $xy+yx=0$.

\item The variety $\Assoc$ of associative algebras is determined by
the identity $x(yz)-(xy)z=0$.

\item The variety $\CAssoc$ of commutative associative algebras is
determined by the identities $xy-yx=0$ and $x(yz)-(xy)z=0$.

\item The variety $\Alt$ of \emph{alternative algebras} is determined
by the identities
\[
(yx)x-y(xx)=0
\quad\text{and}\quad
x(xy)-(xx)y=0.
\]
Every associative algebra is alternative, while the algebra of \emph{octonions} provides an example of an alternative algebra that is not associative.

\item The variety $\Lie$ of \emph{Lie algebras} is characterised by
$xx=0$ together with the \emph{Jacobi identity}
$x(yz)+y(zx)+z(xy)=0$. If $\bchar(\bF)\neq2$, the condition $xx=0$
is equivalent to anti-commutativity. Hence, in this case, the variety
$\Lie$ is operadic.

\item The variety $\JJord$ of \emph{Jacobi--Jordan
algebras}~\cite{Burde} is determined by commutativity and the Jacobi
identity. Jacobi--Jordan algebras, also known as \emph{mock-Lie
algebras}, are the commutative counterpart of Lie algebras.

\item The variety $\ACAAssoc$ of anti-commutative anti-associative
algebras is determined by the identities
$xy+yx=0$ and $x(yz)+(xy)z=0$.

\item The variety $\Leib$ of right Leibniz algebras is determined by
the right Leibniz identity
$(xy)z-(xz)y-x(yz)=0$.

\item The category $\BoolRng$ of not necessarily unital
\emph{Boolean rings} is the variety of commutative
associative $\bF_2$-algebras satisfying $xx=x$. Notice that this
variety is not homogeneous.

\item Given a variety $\V$, one can define its subvariety
$\Nil_k(\V)$ of $k$-\emph{nilpotent} algebras by adjoining to the
identities of $\V$ the identities
\[
x_1x_2\cdots x_{k+1}=0,
\]
for all possible choices of parentheses. For instance, if
$\V=\CAssoc$, every non-trivial multilinear identity of degree
$k+1$ is equivalent to $x_1\cdots x_{k+1}=0$. Hence, the only
proper operadic subvarieties of $\CAssoc$ are the varieties $\Nil_k(\CAssoc)$ of
$k$-nilpotent commutative associative algebras, for $k\geq1$, where $\Nil_1(\CAssoc)=\AbAlg$, together with the trivial variety, which corresponds to $k=0$.
\end{enumerate}
\end{examples}

\subsection{Derived actions in varieties of non-associative algebras}
\label{SecActVar}

We now describe derived actions in an arbitrary variety of non-associative algebras.

\begin{definition}\label{Rem:SplitV}
Let $\V$ be a variety of non-associative algebras, and let
\begin{equation}\label{diag:SplV}
\begin{tikzcd}
X \arrow[r, "k"]
&
A \arrow[r, shift left, "\alpha"]
&
B \arrow[l, shift left, "\beta"]
\end{tikzcd}
\end{equation}
be a split extension in $\V$. The pair of bilinear maps
\[
l\colon B\times X\to X,
\qquad
r\colon X\times B\to X
\]
defined by
\[
b\ast x
=
k^{-1}\bigl(\beta(b)k(x)\bigr)
\quad\text{and}\quad
x\ast b
=
k^{-1}\bigl(k(x)\beta(b)\bigr),
\]
for all $b\in B$ and $x\in X$, where
$b\ast-=l(b,-)$ and $-\ast b=r(-,b)$, is called the
\emph{derived action} of $B$ on $X$ associated
with~\eqref{diag:SplV}.
\end{definition}

Let $B$ and $X$ be objects of $\V$. Given a pair of bilinear maps
\[
l\colon B\times X\to X,
\qquad
r\colon X\times B\to X,
\]
we can define a multiplication on the direct sum of vector spaces
$B\oplus X$ by
\begin{equation}\label{product on semi direct}
(b,x)\cdot(b',x')
=
(bb',xx'+b\ast x'+x\ast b'),
\end{equation}
where $b\ast x'\coloneqq l(b,x')$ and
$x\ast b'\coloneqq r(x,b')$. This construction gives rise to the
split extension in $\Alg$
\begin{equation}
\begin{tikzcd}
X \arrow[r, "\iota_2"]
&
(B\oplus X, \cdot) \arrow[r, shift left, "\pi_1"]
&
B \arrow[l, shift left, "\iota_1"],
\end{tikzcd}
\end{equation}
where $\iota_2(x)=(0,x)$, $\iota_1(b)=(b,0)$, and $\pi_1(b,x)=b$. This is a split extension in~$\V$ if and only if
$(B\oplus X,\cdot)$ is an object of $\V$, that is, if and only if it
satisfies the identities defining $\V$. We therefore obtain the
following result, analogous to~\cite[Part~I, Theorem~2.4]{Orzech}
and~\cite[Lemma~1.8]{Tim}.

\begin{lemma}\label{lemma derived action}
Let $\V$ be a variety of non-associative algebras. A pair of bilinear
maps $(l,r)$ defines a derived action of $B$ on $X$ if and only if
$(B\oplus X,\cdot)$ is an object of $\V$. In this case,
$(B\oplus X,\cdot)$ is called the \emph{semidirect product} of $B$
and $X$ with respect to the derived action $(l,r)$ and is denoted by
$B\ltimes X$.
\end{lemma}

\begin{remark}\label{rem_b*}
For every split extension~\eqref{diag:SplV} and its associated derived
action $(l,r)$, there is an isomorphism of split extensions
\begin{equation*}
\begin{tikzcd}
X
    \arrow[r, "\iota_2"]
    \arrow[d, equal]
&
B\ltimes X
    \arrow[r, shift left, "\pi_1"]
    \arrow[d, "\theta"']
&
B
    \arrow[l, shift left, "\iota_1"]
    \arrow[d, equal]
\\
X \arrow[r, "k"]
&
A \arrow[r, shift left, "\alpha"]
&
B \arrow[l, shift left, "\beta"],
\end{tikzcd}
\end{equation*}
where
$\theta\colon B\ltimes X\to A\colon
(b,x)\mapsto\beta(b)+k(x)$.
Thus, when writing $b\ast x$ (respectively, $x\ast b$), one may
identify it with the product $(b,0)\cdot(0,x)$ (respectively,
$(0,x)\cdot(b,0)$) in $B\ltimes X$.
\end{remark}

Recall that a category $\cC$ with finite limits is \emph{algebraically coherent}~\cite{acc} if, for every morphism $f\colon B'\to B$, the change of base functor $f^*\colon\Pt_B(\cC)\to\Pt_{B'}(\cC)$ between the categories of points over $B$ and~$B'$, whose objects are the split epimorphisms with a chosen section, preserves finite limits and jointly strongly epimorphic pairs.  When~$\cC$ is semi-abelian, it suffices to require this for the kernel functors $\Pt_B(\cC)\to\cC$, that is, for the morphisms $0\to B$; see~\cite[Proposition~3.12]{acc}. We recall the following characterisation.

\begin{theorem}\cite{acc,Tim,GMVdL,Coco_Thesis}\label{Theorem AC iff Orzech}
Let $\V$ be a homogeneous variety of non-associative algebras. The following conditions are equivalent:
\begin{tfae}
\item $\V$ is action accessible.
\item $\V$ is algebraically coherent.
\item $\V$ is an Orzech category of interest.
\item There exist
$\lambda_1,\dots,\lambda_8,\mu_1,\dots,\mu_8\in\bF$ such that
\begin{align*}
x(yz)
={}&\lambda_1(xy)z+\lambda_2(yx)z+\lambda_3z(xy)+\lambda_4z(yx)\\
&+\lambda_5(xz)y+\lambda_6(zx)y+\lambda_7y(xz)+\lambda_8y(zx)
\intertext{and}
(yz)x
={}&\mu_1(xy)z+\mu_2(yx)z+\mu_3z(xy)+\mu_4z(yx)\\
&+\mu_5(xz)y+\mu_6(zx)y+\mu_7y(xz)+\mu_8y(zx)
\end{align*}
are identities of $\V$. These two identities are together called the \emph{$\lambda/\mu$-rules}.
\item $\V$ is a $2$-variety in the sense of~\cite{Zwier}: for any ideal $I$ of an algebra $X$ of~$\V$, the subalgebra $I^2$ of~$X$ is again an ideal.
\end{tfae}
\end{theorem}

We observe that the proof of the equivalence of~(ii) with~(iii), (iv) and~(v) given in~\cite[Theorem~2.12]{GMVdL}, which relies on~\cite{Anderson} for the equivalence of~(iv) and~(v), assumes that the field is infinite only to ensure that the homogeneous components of identities are again identities; see~\cite[Theorem~2.3]{GMVdL}. Hence, it applies verbatim to homogeneous varieties over an arbitrary field.

\begin{remark}\label{USGA_var}
For varieties of non-associative algebras, the existence of the
$\lambda/\mu$-rules is precisely the formulation of Axiom~2 for an
Orzech category of interest. Indeed, in this setting, Axiom~2 requires
the products $x(yz)$ and $(yz)x$ to be expressible as linear
combinations of the eight products appearing on the corresponding
right-hand sides above. Thus, condition~(iv) of \Cref{Theorem AC iff Orzech} is equivalent to Axiom~2. These
expressions need not be unique. Throughout the text, we fix a choice of the $\lambda/\mu$-rules for each variety under consideration.
\end{remark}

\begin{remark}
If $\V$ is a homogeneous, action accessible variety of commutative or anti-commutative algebras, then, by \Cref{Theorem AC iff Orzech}, the $\lambda/\mu$-rules can be chosen of the form
\[
x(yz)=\alpha(xy)z+\beta(xz)y,
\qquad
(yz)x=\varepsilon\alpha(xy)z+\varepsilon\beta(xz)y,
\]
for some $\alpha,\beta\in\bF$, where $\varepsilon=1$ in the commutative case and $\varepsilon=-1$ in the anti-commutative case. Indeed, by (anti-)commutativity, each of the eight products on the right-hand sides of the $\lambda/\mu$-rules is equal to $\pm(xy)z$ or to $\pm(xz)y$, and $(yz)x=\varepsilon\,x(yz)$.
\end{remark}

\begin{theorem}\cite{Tim}\label{AR implies Lie}
Let $\bF$ be an infinite field with $\bchar(\bF) \neq 2$, and let $\V$ be a non-trivial variety of non-associative $\bF$-algebras. The following conditions are equivalent:
\begin{tfae}
\item $\V$ is action representable.
\item $\V$ is either the variety $\Lie$ or the variety $\AbAlg$.
\end{tfae}
\end{theorem}

\Cref{AR implies Lie} shows that, among varieties of non-associative
algebras over an infinite field $\bF$ with $\bchar(\bF) \neq 2$, $\Lie$ is the only non-abelian action
representable variety. Over an infinite field of characteristic~$2$, the non-abelian action representable varieties are $\Lie$ and $\JJord$, the latter coinciding with the variety of \emph{quasi-Lie algebras}, determined by anti-commutativity and the Jacobi identity~\cite[Theorems~1.12 and~6.1]{Tim}. This highlights the importance of studying weak action representability in this setting. The following result shows that every derived action induces a
unique derived action on the image of its acting morphism,
with respect to a fixed weak representation.

\begin{proposition}\cite{WRAAlg}
Let $\V$ be a weakly action representable variety of
non-associative algebras, and let $X$ be an object of $\V$.
Fix a weak representation $(T,\tau)$ of $\SplExt(-,X)$. Given an object $B$ of $\V$ and a derived action
\[
l\colon B\times X\to X,
\qquad
r\colon X\times B\to X,
\]
let $\varphi\colon B\to T$ be the corresponding acting morphism.
Then $\varphi(B)$ admits a unique derived action $(l',r')$ on $X$
such that
\[
l'(\varphi(b),x)=l(b,x)
\quad\text{and}\quad
r'(x,\varphi(b))=r(x,b)
\]
for every $b\in B$ and $x\in X$.
\end{proposition}

\begin{proof}
Write $\varphi=j\circ q$, where
$q\colon B\twoheadrightarrow\varphi(B)$ is the corestriction of
$\varphi$ to its image and
$j\colon\varphi(B)\rightarrowtail T$ is the inclusion.
Since $j\circ q$ is an acting morphism,
\Cref{lemma2WRAimpliesAA} implies that $j$ is also an acting
morphism. Let $(l',r')$ be its corresponding derived action
on $X$.

Naturality and injectivity of $\tau$ show that the original
split extension is isomorphic to the pullback of the extension
corresponding to $(l',r')$ along $q$. Hence
\[
\begin{aligned}
l'(\varphi(b),x)&=l'(q(b),x)=l(b,x),\\
r'(x,\varphi(b))&=r'(x,q(b))=r(x,b),
\end{aligned}
\]
for every $b\in B$ and $x\in X$. These equalities determine $(l',r')$ uniquely, since $q$ is surjective.
\end{proof}

\subsection{Partial algebras}\label{SecPart}

We conclude this section by introducing partial algebras,
using a modified formulation of the definition in~\cite{WRAAlg}.

There, the domain $\Omega\subseteq X\times X$ of a partial
multiplication was required to be a vector subspace.
In the constructions under consideration, products with zero are always defined, so $(x,0),(0,y)\in\Omega$ for all $x,y\in X$. The subspace requirement then gives
\[
(x,y)=(x,0)+(0,y)\in\Omega
\]
for every $x,y \in X$. Thus $\Omega=X\times X$, and the multiplication is necessarily total. To allow non-total multiplications, we instead take the domain to be a vector subspace of $X\otimes X$, as follows.

\begin{definition}
Let $X$ be an $\bF$-vector space. A \emph{bilinear partial operation}
on $X$ is a pair $(D,\mu)$ consisting of a vector subspace
\[
D\leq X\otimes X
\]
and a linear map
\[
\mu\colon D\to X.
\]
Given $x,y\in X$, we say that the product $xy$ is
\emph{defined} if $x\otimes y\in D$. In this case, we set
\[
xy\coloneqq\mu(x\otimes y).
\]
\end{definition}

\begin{definition}
A \emph{partial algebra} over $\bF$ is a triple $(X,D,\mu)$, where
$X$ is an $\bF$-vector space and $(D,\mu)$ is a bilinear partial
operation on $X$.

The partial algebra $(X,D,\mu)$ is said to be \emph{total} if $D=X\otimes X$. In this case, $\mu$ is the linearisation of an ordinary bilinear
multiplication on $X$.
\end{definition}

\begin{definition}
Let $(X,D,\mu)$ and $(X',D',\mu')$ be partial algebras over $\bF$.
A \emph{homomorphism of partial algebras} is a linear map
$f\colon X\to X'$ such that
\[
(f\otimes f)(D)\subseteq D' \quad \text{ and } \quad
f\circ\mu=\mu'\circ\left.(f\otimes f)\right|_D.
\]
\end{definition}

We denote by $\PAlg$ the category whose objects are partial algebras
over $\bF$ and whose morphisms are homomorphisms of partial algebras.

\begin{remark}\label{remzeroPAlg}
The category $\PAlg$ is pointed. Its zero object is the zero vector space endowed with its unique bilinear partial operation, which is total, and we denote it by~$0$, as usual. Indeed, let $(X,D,\mu)$ be a partial algebra. Since $D$ is a vector subspace of $X\otimes X$, it contains the zero tensor, so that the unique linear map from the zero vector space to $X$ is a homomorphism of partial algebras. The unique linear map from $X$ to the zero vector space is clearly a homomorphism of partial algebras as well. Consequently, the zero morphism between two partial algebras is the zero linear map. Moreover, for every $x\in X$, the tensors $x\otimes0$ and $0\otimes x$ are zero and hence belong to~$D$, so that the products $x0$ and $0x$ are always defined and equal to~$0$.
\end{remark}

\begin{definition}
Let $(X,D,\mu)$ be a partial algebra, and let
\[
t\in\bFM(\{x_1,\ldots,x_n\})
\]
be a non-associative word. For every $\mathbf a=(a_1,\ldots,a_n)\in X^n$, we define the \emph{partial evaluation} of $t$ at $\mathbf a$
recursively as follows.

If $t=x_i$, then $t$ is defined at $\mathbf a$, and $t(\mathbf a)=a_i$. 

If $t=t_1t_2$, assume that $t_1$ and $t_2$ are defined at $\mathbf a$, with respective values
\[
t_1(\mathbf a)=b
\qquad\text{and}\qquad
t_2(\mathbf a)=c.
\]
Then $t$ is defined at $\mathbf a$ if and only if $b\otimes c\in D$, and in this case
\[
t(\mathbf a)
=
\mu(b\otimes c).
\]
If one of these conditions is not satisfied, then $t$ is not defined at $\mathbf a$.

Let $p\in\bF[\bFM(\{x_1,\ldots,x_n\})]$ have the unique expansion
\[
p=\sum_{j=1}^{r}\lambda_jt_j,
\]
where the $t_j$ are pairwise distinct non-associative words and
$\lambda_j\in\bF\setminus\{0\}$.
We say that $p$ is \emph{defined} at $\mathbf a$ if every $t_j$ is
defined at $\mathbf a$. In this case, we set
\[
p(\mathbf a)=\sum_{j=1}^{r}\lambda_jt_j(\mathbf a).
\]
We allow $r=0$; thus the zero polynomial is defined at every $\mathbf a\in X^n$ and has value zero.

We say that $(X,D,\mu)$ \emph{satisfies the polynomial identity}
\[
p(x_1,\ldots,x_n)=0
\]
if $p(\mathbf a)=0$ for every $\mathbf a\in X^n$ at which $p$ is defined.
\end{definition}

\begin{example}
Let $x,y,z\in X$. The monomial $x(yz)$ is defined at
$(x,y,z)$ if and only if
\[
y\otimes z\in D \quad \text{ and } \quad x\otimes\mu(y\otimes z)\in D.
\]
When these conditions hold, its value is
\[
x(yz)
=
\mu\bigl(
x\otimes\mu(y\otimes z)
\bigr).
\]

Similarly, the monomial $(xy)z$ is defined if and only if
\[
x\otimes y\in D \quad \text{ and } \quad \mu(x\otimes y)\otimes z\in D,
\]
and in this case
\[
(xy)z
=
\mu\bigl(
\mu(x\otimes y)\otimes z
\bigr).
\]
\end{example}

\begin{example}
A partial algebra $(X,D,\mu)$ is \emph{associative} if
\[
x(yz)=(xy)z
\]
for every $x,y,z\in X$ for which both sides are defined.
\end{example}

\begin{remark}
Let $X$ be a non-associative algebra, and let
\[
\widetilde m_X\colon X\otimes X\to X
\]
be the linearisation of its multiplication. Then
\[
\widetilde U(X)
=
(X,X\otimes X,\widetilde m_X)
\]
is a total algebra. This construction defines a canonical fully faithful functor $\widetilde U\colon\V\to\PAlg$ for every variety of non-associative algebras $\V$. In particular, $\widetilde U$ sends the zero algebra to the zero object of~$\PAlg$.
\end{remark}

\begin{lemma}\label{lem:morphism-from-total}
Let $B$ be a non-associative algebra, let $(X,D,\mu)$ be a partial
algebra, and let $f\colon B\to X$ be a linear map. Then
\[
f\colon\widetilde U(B)\to(X,D,\mu)
\]
is a homomorphism of partial algebras if and only if, for every
$b,b'\in B$,
\[
f(b)\otimes f(b')\in D
\]
and
\[
f(bb')
=
\mu\bigl(f(b)\otimes f(b')\bigr).
\]
\end{lemma}

\begin{proof}
The pure tensors $b\otimes b'$ span $B\otimes B$. Therefore,
\[
(f\otimes f)(B\otimes B)\subseteq D
\]
if and only if $f(b)\otimes f(b')\in D$ for all $b,b'\in B$.
Likewise, the equality
\[
f\circ\widetilde m_B
=
\mu\circ(f\otimes f)
\]
between linear maps $B\otimes B\to X$ may be checked on pure tensors.
\end{proof}

\begin{remark}\label{rem:pure}
The domain $D$ of a partial algebra $(X,D,\mu)$ may contain tensors which are not linear combinations of pure tensors belonging to~$D$. Such tensors do not affect the defined products, but they do affect the homomorphisms of partial algebras. For example, let $X=\bF^2$, with canonical basis $\{e_1,e_2\}$, and let $D$ be the subspace of $X\otimes X$ spanned by $e_1\otimes e_2+e_2\otimes e_1$. Since this tensor has rank~$2$,~$D$ contains no non-zero pure tensors. Hence the partial algebras $(X,D,0)$ and $(X,0,0)$ have the same defined products, namely those involving~$0$, but they are not isomorphic in~$\PAlg$, since an isomorphism $f$ would satisfy $(f\otimes f)(D)=0$, whereas $f\otimes f$ is injective. This does not affect the results of this survey: actions are represented by means of homomorphisms of partial algebras whose domain is a total algebra (see \Cref{thmV}), and, by \Cref{lem:morphism-from-total}, these only depend on the defined products of the codomain.
\end{remark}

\begin{lemma}\label{lem:partial-from-bilinear-map}
Let $W$ be a vector subspace of a vector space $Z$, and let
\[
m\colon W\times W\to Z
\]
be a bilinear map. Denote its linearisation by $\widetilde m\colon W\otimes W\to Z$. Then
\[
D_m\coloneqq\widetilde m^{-1}(W)
\]
is a vector subspace of $W\otimes W$, and
\[
\mu_m\coloneqq\left.\widetilde m\right|_{D_m}
\colon D_m\to W
\]
defines a partial algebra $(W,D_m,\mu_m)$. For $u,v\in W$, the product
$u\cdot v$ is defined precisely when $m(u,v)\in W$, and in this case
\[
u\cdot v=m(u,v).
\]
\end{lemma}

\begin{proof}
Since $W\leq Z$ and $\widetilde m$ is linear, the inverse image
$D_m=\widetilde m^{-1}(W)$ is a vector subspace of $W\otimes W$.
The remaining assertions follow immediately from the definitions.
\end{proof}

\section{Representability of actions of non-associative algebras}
\label{SecRepVar}

Throughout this section, we assume that $\V$ is an action accessible,
operadic variety of non-associative algebras. We fix a defining family of multilinear identities
\[
\Phi_{k,i}(x_1,\ldots,x_k)=0,
\qquad 1\leq i\leq n_k,
\]
where $k=\deg \Phi_{k,i} \geq 2$ (an identity of degree~$1$ only occurs in the trivial variety, which we exclude). For each $k$, the integer $n_k$ may be chosen finite, since the space of multilinear polynomials in $k$ fixed variables is finite-dimensional. We also fix
$\lambda_1,\ldots,\lambda_8,\mu_1,\ldots,\mu_8\in\bF$ determining
a choice of the $\lambda/\mu$-rules.

Our aim is to study the representability of actions in $\V$. More precisely, we shall see that, for every object $X$ of $\V$, there
exists a vector space $\E(X)$ satisfying
\[
\Inn(X)\leq\E(X)\leq\End(X)^2,
\]
where
\[
\Inn(X)=\{(L_x,R_x)\mid x\in X\}
\]
is the vector space of left and right multiplications of $X$, given by $L_x(y)=xy$ and $R_x(y)=yx$ for all $y\in X$. Moreover, once the $\lambda/\mu$-rules have been fixed, $\E(X)$ carries a partial algebra structure
\[
(\E(X),D_X,\mu_X)
\]
in such a way that, for any split extension of $B$ by $X$ in $\V$, one can naturally associate a homomorphism of partial algebras
\[
\widetilde U(B)\to(\E(X),D_X,\mu_X).
\]

\begin{proposition}\label{prop_semidirect}\cite{WRAAlg}
Let $X$ and $B$ be objects of $\V$. Given a pair of bilinear maps
\[
l\colon B\times X\to X,
\qquad
r\colon X\times B\to X,
\]
we construct $(B\oplus X,\cdot)$ as in
\eqref{product on semi direct}. Then $(B\oplus X,\cdot)$ is an object
of $\V$ if and only if
\[
\Phi_{k,i}(\alpha_1,\ldots,\alpha_k)=0
\]
for every $i=1,\ldots,n_k$ and every mixed choice of
$\alpha_1,\ldots,\alpha_k$, where each $\alpha_j$ is either of the
form $(b,0)$, with $b\in B$, or of the form $(0,x)$, with $x\in X$,
and at least one element of each form occurs.
\end{proposition}

Using the notation of \Cref{rem_b*}, we obtain the following result.

\begin{corollary}\cite{WRAAlg}
Suppose that every identity of $\V$ is a consequence of the
$\lambda/\mu$-rules. Then $(B\oplus X,\cdot)$ is an object of $\V$ if
and only if the following identities hold:
\begin{enumerate}
\item[(1)]
$b\ast(xx')
 =\lambda_1(b\ast x)x'+\cdots+\lambda_8x(x'\ast b)$;

\item[(2)]
$(xx')\ast b
 =\mu_1(b\ast x)x'+\cdots+\mu_8x(x'\ast b)$;

\item[(3)]
$x(x'\ast b)
 =\lambda_1(xx')\ast b+\cdots+\lambda_8x'(b\ast x)$;

\item[(4)]
$(x'\ast b)x
 =\mu_1(xx')\ast b+\cdots+\mu_8x'(b\ast x)$;

\item[(5)]
$x(b\ast x')
 =\lambda_1(x\ast b)x'+\cdots+\lambda_8b\ast(x'x)$;

\item[(6)]
$(b\ast x')x
 =\mu_1(x\ast b)x'+\cdots+\mu_8b\ast(x'x)$;

\item[(7)]
$x\ast(bb')
 =\lambda_1(x\ast b)\ast b'
  +\cdots+\lambda_8b\ast(b'\ast x)$;

\item[(8)]
$(bb')\ast x
 =\mu_1(x\ast b)\ast b'
  +\cdots+\mu_8b\ast(b'\ast x)$;

\item[(9)]
$b\ast(b'\ast x)
 =\lambda_1(bb')\ast x+\cdots+\lambda_8b'\ast(x\ast b)$;

\item[(10)]
$(b'\ast x)\ast b
 =\mu_1(bb')\ast x+\cdots+\mu_8b'\ast(x\ast b)$;

\item[(11)]
$b\ast(x\ast b')
 =\lambda_1(b\ast x)\ast b'
  +\cdots+\lambda_8x\ast(b'b)$;

\item[(12)]
$(x\ast b')\ast b
=\mu_1(b\ast x)\ast b'
 +\cdots+\mu_8 x\ast(b'b)$;
 \end{enumerate}
for every $b,b'\in B$ and $x,x'\in X$.
\end{corollary}

\begin{definition}\label{E(X)}\cite{WRAAlg}
Let $X$ be an object of $\V$. We define $\E(X)$ as the subspace of
all pairs
\[
f=(f\ast-,-\ast f)\in\End(X)^2
\]
satisfying
\[
\Phi_{k,i}(\alpha_1,\ldots,\alpha_k)=0,
\qquad
i=1,\ldots,n_k,
\]
for every $k$, every $j\in\{1,\ldots,k\}$ and all $\alpha_t\in X$ with $t\neq j$, where $\alpha_j=f$. Here, we write
\[
fx\coloneqq f\ast x,
\qquad
xf\coloneqq x\ast f.
\]
We consider the bilinear map
\[
m_X\colon
\E(X)\times\E(X)\to\End(X)^2
\]
defined by
\[
m_X\left(
(f\ast-,-\ast f),(g\ast-,-\ast g)
\right)
=
(h\ast-,-\ast h),
\]
where
\begin{align*}
x\ast h
={}&\lambda_1(x\ast f)\ast g
  +\lambda_2(f\ast x)\ast g
  +\lambda_3g\ast(x\ast f)
  +\lambda_4g\ast(f\ast x)\\
 &+\lambda_5(x\ast g)\ast f
  +\lambda_6(g\ast x)\ast f
  +\lambda_7f\ast(x\ast g)
  +\lambda_8f\ast(g\ast x)
\intertext{and}
h\ast x
={}&\mu_1(x\ast f)\ast g
  +\mu_2(f\ast x)\ast g
  +\mu_3g\ast(x\ast f)
  +\mu_4g\ast(f\ast x)\\
 &+\mu_5(x\ast g)\ast f
  +\mu_6(g\ast x)\ast f
  +\mu_7f\ast(x\ast g)
  +\mu_8f\ast(g\ast x).
\end{align*}
\end{definition}

\begin{remark}
The subspace $\E(X)$ does not depend on the chosen defining family of identities. Indeed, let $\psi(x_1,\ldots,x_n)=0$ be a multilinear identity of $\V$. Since $\V$ is operadic, $\psi$ is a linear combination of terms obtained from some $\Phi_{k,i}(m_1,\ldots,m_k)$, where $m_1,\ldots,m_k$ are monomials, by multiplying it repeatedly, on the left or on the right, by further monomials. Since $\psi$ is multilinear, we may assume that each variable occurs exactly once in each such term, so that all the monomials involved have pairwise disjoint sets of variables. Replace exactly one variable $x_j$ by $f\in\E(X)$ and the others by elements of~$X$, and consider one of these terms. If $x_j$ occurs in one of the outer monomials, then $\Phi_{k,i}(m_1,\ldots,m_k)$ is evaluated in $X$, hence it vanishes since $X$ is an object of $\V$, and so does the whole term. If $m_l=x_j$ for some $l$, then $\Phi_{k,i}(m_1,\ldots,m_k)$ vanishes by the definition of $\E(X)$, since its other arguments are evaluated in~$X$. Finally, if $x_j$ occurs in $m_l$ and $m_l\neq x_j$, then $m_l$ is evaluated to an element of $X$, so that all the arguments of $\Phi_{k,i}$ lie in $X$, and the term vanishes again. Hence every multilinear identity of $\V$ vanishes when exactly one of its variables is replaced by $f\in\E(X)$ and the others by elements of~$X$, which proves the claim.
\end{remark}

\begin{remark}
Suppose that every identity of $\V$ is a consequence of the
$\lambda/\mu$-rules. Then $\E(X)$ is the subspace of all pairs
$(f\ast-,-\ast f)\in\End(X)^2$ satisfying
\begin{enumerate}
\item[(1)]
$f\ast(xx')
 =\lambda_1(f\ast x)x'+\cdots+\lambda_8x(x'\ast f)$;

\item[(2)]
$(xx')\ast f
 =\mu_1(f\ast x)x'+\cdots+\mu_8x(x'\ast f)$;

\item[(3)]
$x(x'\ast f)
 =\lambda_1(xx')\ast f+\cdots+\lambda_8x'(f\ast x)$;

\item[(4)]
$(x'\ast f)x
 =\mu_1(xx')\ast f+\cdots+\mu_8x'(f\ast x)$;

\item[(5)]
$x(f\ast x')
 =\lambda_1(x\ast f)x'+\cdots+\lambda_8f\ast(x'x)$;

\item[(6)]
$(f\ast x')x
 =\mu_1(x\ast f)x'+\cdots+\mu_8f\ast(x'x)$;
\end{enumerate}
for any $x,x'\in X$.
\end{remark}

\begin{remark}
The choice of the $\lambda/\mu$-rules does not affect the underlying
vector space of $\E(X)$, but it plays an essential role in the
definition of the bilinear map $m_X$ and hence in the partial
algebra structure on $\E(X)$.
\end{remark}

\begin{remark}
In general, the bilinear map
\[
m_X\colon\E(X)\times\E(X)\to\End(X)^2
\]
does not corestrict to $\E(X)$, i.e., it does not define a bilinear operation on $\E(X)$. In other words, there may exist
$f,g\in\E(X)$ such that
\[
m_X(f,g)\notin\E(X).
\]
Even when $m_X$ corestricts to a bilinear operation on
$\E(X)$, the resulting total algebra may fail to satisfy
some of the defining identities of $\V$.
\end{remark}

\begin{example}\label{Example associative}
Let $\V=\Assoc$, and let $X$ be an associative algebra. Then $\E(X)\cong\Bim(X)$ as vector spaces, where $\Bim(X)$ is the associative algebra of bimultipliers of $X$ described in~\Cref{exassoc}. If we consider the standard choice of the $\lambda/\mu$-rules given by
\[
\lambda_1=\mu_8=1
\]
and
\[
\lambda_i=\mu_j=0
\]
for every $i\in\{2,\ldots,8\}$ and $j\in\{1,\ldots,7\}$, then the above isomorphism is also an isomorphism of associative algebras.
\end{example}

\begin{example}\label{exLeib}
Let $\V=\Leib$, and let $X$ be a Leibniz algebra. Then
\[
\E(X)
=
\left\{
(f\ast-,-\ast f)\in\End(X)^2
\,\middle|\,
\begin{aligned}
f\ast[x,y]
  &= [f\ast x,y]-[f\ast y,x],\\
[x,y]\ast f
  &= [x\ast f,y]+[x,y\ast f],\\
[x,f\ast y]&+[x,y\ast f]
  =0
\end{aligned}
\, \forall x,y\in X
\right\}.
\]
With the choice of the $\lambda/\mu$-rules as in \Cref{esLeib}, we have that $\E(X)$ is isomorphic to the Leibniz algebra
\[
\Bider(X)
=
\left\{
(d,D)\in\End(X)^2
\,\middle|\,
\begin{aligned}
d([x,y])
  &= [d(x),y]+[x,d(y)],\\
D([x,y])
  &= [D(x),y]-[D(y),x],\\
[x,d(y)]
  &= [x,D(y)]
\end{aligned}
\quad \forall x,y\in X
\right\}
\]
of \emph{biderivations} of $X$~\cite{loday}, whose bracket is given by
\[
\left[(d,D),(d',D')\right]
=
\bigl(
d\circ d'-d'\circ d,\,
D\circ d'-d'\circ D
\bigr).
\]
The isomorphism $\Bider(X) \to \E(X)$ is given by the linear map
\[
(d,D) \mapsto (D,-d).
\]
On the other hand, if we choose the $\lambda/\mu$-rules as
\begin{align*}
[x,[y,z]]
  &= [[x,y],z]-[[x,z],y],\\
[[y,z],x]
  &= [[y,x],z]+[y,[z,x]],
\end{align*}
the underlying vector space of $\E(X)$, identified with $\Bider(X)$ via the above isomorphism, is endowed with the non-associative operation defined in~\cite[Definition~5.2]{Casas}:
\[
\bigl\{(d,D),(d',D')\bigr\}
=
\bigl(
d\circ d'-d'\circ d,\,
D\circ D'-d'\circ D
\bigr).
\]
This operation does not, in general, satisfy the right Leibniz
identity and therefore does not define a Leibniz algebra structure on
$\E(X)$. For instance, if $X=\bF$ is the one-dimensional abelian algebra, then $\Bider(X) \cong \bF^2$, and $u=(0,1)$ satisfies $\{u,u\}=u$, so that $\{u,\{u,u\}\}=u\neq0$, whereas $[x,[y,y]]=0$ for all elements $x,y$ of a right Leibniz algebra.
This is consistent with~\cite[Proposition~5.4]{Casas}, where the universal strict general actor associated with the presentation of~\cite[Definition~5.2]{Casas} is shown to be a Leibniz algebra for the Leibniz algebras $X$ satisfying~\cite[Condition~1]{Casas}, namely $c\ast(x\ast b)=-c\ast(b\ast x)$ for all $x\in X$ and all elements $b$ and $c$ of Leibniz algebras acting on~$X$. This condition fails for $X=\bF$: the abelian algebras $B=\bF\{b\}$ and $C=\bF\{c\}$ act on $X$ by $b\ast x=0$, $x\ast b=x$, $c\ast x=x$ and $x\ast c=-x$, as one checks directly from the mixed Leibniz identities, and $c\ast(x\ast b)=x\neq0=-c\ast(b\ast x)$ for every $x\neq0$.
\end{example}

\begin{example}\label{exnil2}
Let $\V=\Nil_2(\Alg)$ be the variety of $2$-nilpotent algebras
(see~\cite{LaRosaMancini1,LaRosaMancini2}), determined
by the identities
\[
x(yz)=(yz)x=0.
\]
Then
\[
\E(X)
=
\left\{
(f\ast-,-\ast f)\in\End(X)^2
\,\middle|\,
\begin{aligned}
f\ast(xy)&=(xy)\ast f=0,\\
(f\ast x)y&=x(y\ast f)=0,\\
x(f \ast y)&=(x \ast f)y=0,
\end{aligned}
\quad \forall x,y\in X
\right\},
\]
and, choosing $\lambda_i=\mu_i=0$ for $i=1,\ldots,8$, the bilinear map $m_X$ is given by
\[
m_X\left((f\ast-,-\ast f),(g\ast-,-\ast g)\right)=(0,0).
\]
Thus, $\E(X)$ is an abelian algebra.
\end{example}

\begin{example}\label{Example Nilk}
Let $\V=\Nil_k(\Assoc)$, with $k\geq3$. Then
\[
\E(X)
=
\left\{
(f\ast-,-\ast f)\in\Bim(X)
\;\middle|\;
\begin{aligned}
&f\ast(x_1\cdots x_k)=0=(x_1\cdots x_k)\ast f\\
&(x_1\cdots x_j)\bigl(f\ast(x_{j+1}\cdots x_k)\bigr)=0
\end{aligned}\right\}
\]
where the conditions hold for every $ x_1,\ldots,x_k \in X$ and $1 \leq j < k$.
With the same choice of the $\lambda/\mu$-rules as in
\Cref{Example associative}, the bilinear map $m_X$ becomes
\[
m_X\left(
(f\ast-,-\ast f),(g\ast-,-\ast g)
\right)
=
\bigl(f\ast(g\ast-),(-\ast f)\ast g\bigr).
\]
The subspace $\E(X)$ is closed under $m_X$. Indeed, let $f,g\in\E(X)$ and $h=m_X(f,g)$, so that $h\ast x=f\ast(g\ast x)$ and $x\ast h=(x\ast f)\ast g$ for all $x\in X$. Then $h\in\Bim(X)$, $h\ast(x_1\cdots x_k)=f\ast\bigl(g\ast(x_1\cdots x_k)\bigr)=0$ and $(x_1\cdots x_k)\ast h=\bigl((x_1\cdots x_k)\ast f\bigr)\ast g=0$. Moreover, for $1\leq j<k$, the defining identities of bimultipliers give
\begin{align*}
(x_1\cdots x_j)\bigl(h\ast(x_{j+1}\cdots x_k)\bigr)
&=\bigl((x_1\cdots x_j)\ast f\bigr)\bigl(g\ast(x_{j+1}\cdots x_k)\bigr)\\
&=\bigl(x_1\cdots x_{j-1}(x_j\ast f)\bigr)\bigl(g\ast(x_{j+1}\cdots x_k)\bigr),
\end{align*}
which vanishes by the condition on $g$, since $x_1\cdots x_{j-1}(x_j\ast f)$ is a product of $j$ elements of $X$ (for $j=1$, it is $x_1\ast f$). This makes $\E(X)$ an associative algebra, although it need not be
$k$-nilpotent. For instance, let $X=\bF$ be the one-dimensional
abelian algebra. Then
\[
\E(X)
=
\End(X)\times\End(X)^{\op}
\cong\bF^2,
\]
which is not nilpotent. Indeed, every linear endomorphism of $X$ is of
the form
\[
\varphi_\alpha\colon x\mapsto\alpha x
\]
for some $\alpha\in\bF$, and
\[
m_X\left(
(\varphi_\alpha,\varphi_\beta),
(\varphi_{\alpha'},\varphi_{\beta'})
\right)
=
(\varphi_\alpha\circ\varphi_{\alpha'},
 \varphi_{\beta'}\circ\varphi_\beta)
=
(\varphi_{\alpha\alpha'},\varphi_{\beta'\beta}).
\]
\end{example}

\begin{example}\label{exalt}
Let $\bchar(\bF)\neq2$ and consider the variety $\Alt$ of alternative algebras. The multilinearisation process shows that $\Alt$ is
equivalently determined by the identities
\[
(xy)z+(xz)y-x(yz)-x(zy)=0
\]
and
\[
(xy)z+(yx)z-x(yz)-y(xz)=0.
\]
These identities are obtained by linearising those of \Cref{Examples varieties}(5). Conversely, setting $z=y$ in the first one and $y=x$ in the second one gives twice the right and the left alternative identities, respectively. This is where the assumption $\bchar(\bF)\neq2$ is used.
A choice of the $\lambda/\mu$-rules is given by
\[
\lambda_1=\lambda_2=-\lambda_7=-\mu_2=\mu_7=\mu_8=1,
\]
with all the remaining coefficients equal to zero. In this case,
$\E(X)$ consists of the pairs
\[
(f\ast-,-\ast f)\in\End(X)^2
\]
satisfying
\[
f\ast(xy)
=
(x\ast f)y+(f\ast x)y-x(f\ast y),
\]
\[
(xy)\ast f
=
x(f\ast y)+x(y\ast f)-(x\ast f)y,
\]
\[
x(y\ast f)
=
(yx)\ast f+(xy)\ast f-y(x\ast f),
\]
and
\[
(f\ast x)y
=
f\ast(yx)+f\ast(xy)-(f\ast y)x
\]
for all $x,y\in X$. The bilinear map
\[
m_X\left(
(f\ast-,-\ast f),(g\ast-,-\ast g)
\right)
=
(h\ast-,-\ast h)
\]
is given by
\[
h\ast x
=
-(f\ast x)\ast g+f\ast(g\ast x)+f\ast(x\ast g)
\]
and
\[
x\ast h
=
(x\ast f)\ast g+(f\ast x)\ast g-f\ast(x\ast g).
\]
Even when $m_X$ defines a bilinear operation on~$\E(X)$, this multiplication need not be alternative. For example, let $X=\bF^2$ be the abelian $2$-dimensional algebra. Then, as vector spaces, $\E(X)=\End(X)^2 \cong \M_2(\bF)^2$, and $m_X$ is total. Identifying $\End(X)$ with $\M_2(\bF)$ via the canonical basis of $\bF^2$, so that composition corresponds to the product of matrices, set
\[
P=\begin{pmatrix}1&0\\0&0\end{pmatrix},\qquad
Q=\begin{pmatrix}0&1\\0&0\end{pmatrix},\qquad
f=(P,0),\quad g=(0,Q).
\]
Writing $\langle u,v\rangle=m_X(u,v)$, direct calculation gives
\[
\langle f,f\rangle=f,\qquad
\langle f,g\rangle=(Q,-Q),\qquad
\langle f,\langle f,g\rangle\rangle=(0,Q).
\]
Thus
\[
\langle\langle f,f\rangle,g\rangle
=(Q,-Q)\neq(0,Q)
=\langle f,\langle f,g\rangle\rangle,
\]
so the left alternative identity fails.
\end{example}

Let
\[
\widetilde m_X\colon
\E(X)\otimes\E(X)\longrightarrow\End(X)^2
\]
be the linearisation of the bilinear map $m_X$. Define
\[
D_X
\coloneqq
\widetilde m_X^{-1}\bigl(\E(X)\bigr)
\leq
\E(X)\otimes\E(X)
\]
and
\[
\mu_X
\coloneqq
\left.\widetilde m_X\right|_{D_X}
\colon
D_X\longrightarrow\E(X).
\]
By \Cref{lem:partial-from-bilinear-map},
\[
\bigl(\E(X),D_X,\mu_X\bigr)
\]
is a partial algebra. By abuse of notation, we shall continue to
denote this partial algebra simply by $\E(X)$.

The partial multiplication induced by $\mu_X$ will be denoted by $\langle-,-\rangle$. More precisely, for $f,g\in\E(X)$, the product
$\langle f,g\rangle$ is defined if and only if
\[
f\otimes g\in D_X,
\]
or, equivalently, if and only if
\[
m_X(f,g)\in\E(X).
\]
Whenever it is defined, we have
\[
\langle f,g\rangle
=
\mu_X(f\otimes g)
=
m_X(f,g).
\]

Notice that $D_X$ contains the kernel of $\widetilde m_X$, which in general contains tensors that are not pure, so that the phenomenon described in \Cref{rem:pure} may occur. This does not affect the homomorphisms $\widetilde U(B)\to\E(X)$ used below to represent actions, since, by \Cref{lem:morphism-from-total}, they only depend on the defined products of~$\E(X)$.

\begin{remark}
The construction of $\E(X)$ gives an alternative characterisation of
split extensions in $\V$. A split extension of $B$ by $X$ in $\V$ is
equivalent to a homomorphism of partial algebras
\[
\varphi\colon
\widetilde U(B)\to\E(X) \colon b\mapsto(b\ast-,-\ast b),
\]
such that
\[
\Phi_{k,i}(\alpha_1,\ldots,\alpha_k)=0,
\qquad
i=1,\ldots,n_k,
\]
for all the mixed choices described in \Cref{prop_semidirect}.
\end{remark}

We thus obtain the following result, which adapts~\cite{WRAAlg} to the notion of partial algebra adopted in \Cref{SecPart}.

\begin{theorem}\label{thmV}\cite{WRAAlg}
Let $\V$ be an action accessible, operadic variety of non-associative
algebras.

\begin{enumerate}
\item[(i)]
For every object $X$ of $\V$, there exists a natural monomorphism of
functors
\[
\tau\colon
\SplExt(-,X)
\rightarrowtail
\Hom_{\PAlg}(\widetilde U(-),\E(X)).
\]
For every object $B$ of $\V$, the component $\tau_B$ sends the isomorphism class of a split extension of $B$ by $X$, with associated derived action $(l,r)$, to the homomorphism of partial algebras
\[
\widetilde U(B)\to\E(X) \colon b\mapsto(b\ast-,-\ast b).
\]

\item[(ii)]
Let $B$ and $X$ be objects of $\V$. A homomorphism of partial algebras
\[
\varphi\colon
\widetilde U(B)\to\E(X) \colon b\mapsto(b\ast-,-\ast b)
\]
belongs to $\operatorname{Im}(\tau_B)$ if and only if
\[
\Phi_{k,i}(\alpha_1,\ldots,\alpha_k)=0
\]
for all the mixed choices described in \Cref{prop_semidirect}.

\item[(iii)]
If $D_X=\E(X)\otimes\E(X)$ and the resulting total algebra, with multiplication
\[
\langle f,g\rangle=\mu_X(f\otimes g),
\]
is an object of $\V$, then $(\E(X),\tau)$ is a weak representation of
$\SplExt(-,X)$.

\item[(iv)]
Suppose that $\V$ is a variety of commutative (respectively, anti-commutative)
algebras, let $\varepsilon=1$ (respectively, $\varepsilon=-1$), and choose the $\lambda/\mu$-rules in the form
\[
\begin{aligned}
x(yz)&=\alpha(xy)z+\beta(xz)y,\\
(yz)x&=\varepsilon\alpha(xy)z+\varepsilon\beta(xz)y,
\end{aligned}
\]
for some $\alpha,\beta\in\bF$.

Let $\A_X$ consist of the linear endomorphisms $f\colon X\to X$ for which
\[
\Phi_{k,i}(x_1,\ldots,x_{j-1},f,x_{j+1},\ldots,x_k)=0
\]
for every defining identity $\Phi_{k,i}$, every $1\leq j\leq k$,
and every choice of the other variables in $X$. Products involving
$f$ are evaluated using $fx=f(x)$ and $xf=\varepsilon f(x)$. Consider the bilinear map
\[
m_X^\varepsilon\colon \A_X\times \A_X\to\End(X) \colon (f,g)
\mapsto \varepsilon\bigl(\alpha(g\circ f)+\beta(f\circ g)\bigr),
\]
and let
\[
\widetilde m_X^\varepsilon\colon \A_X\otimes \A_X\to\End(X)
\]
be its linearisation. Define
\[
D_X^\varepsilon
=\left(\widetilde m_X^\varepsilon\right)^{-1}(\A_X),
\qquad
\mu_X^\varepsilon
=\left.\widetilde m_X^\varepsilon\right|_{D_X^\varepsilon}.
\]
Then the linear map $j_X\colon \A_X\to\E(X)$, defined by
$j_X(f)=(f,\varepsilon f)$, is an isomorphism of partial algebras
\[
(\A_X,D_X^\varepsilon,\mu_X^\varepsilon)
\cong(\E(X),D_X,\mu_X).
\]
\end{enumerate}
\end{theorem}

\begin{proof}{\ }
\begin{itemize}
\item[(i)] Let $(l,r)$ be the derived action associated with a split extension of $B$ by $X$, and set $\varphi(b)=(b\ast-,-\ast b)$. \Cref{prop_semidirect}, applied to the mixed choices containing exactly one element of $B$, gives $\varphi(b)\in\E(X)$. Applying the fixed $\lambda/\mu$-rules in the semidirect product gives
$m_X(\varphi(b),\varphi(b'))=\varphi(bb')\in\E(X)$; hence
$\varphi(b)\otimes\varphi(b')\in D_X$ and
$\mu_X(\varphi(b)\otimes\varphi(b'))=\varphi(bb')$, so $\varphi$ is a
homomorphism of partial algebras by \Cref{lem:morphism-from-total}.
Since $(l,r)$ is recovered from $\varphi$, each $\tau_B$ is injective;
naturality follows from the fact that change of base along
$g\colon B'\to B$ acts on derived actions by precomposition with~$g$.

\item[(ii)] If $\varphi$ satisfies the mixed identities, then
$(B\oplus X,\cdot)$ is an object of $\V$ by \Cref{prop_semidirect},
and $\varphi$ is the image of the corresponding split extension. The
converse is clear.

\item[(iii)] If $D_X=\E(X)\otimes\E(X)$, then
$(\E(X),D_X,\mu_X)=\widetilde U(\E(X))$, where $\E(X)$ is regarded as
an algebra with multiplication~$m_X$. Since $\widetilde U$ is fully
faithful, if this algebra lies in $\V$ we get
$\Hom_{\PAlg}(\widetilde U(-),\E(X))\cong\Hom_{\V}(-,\E(X))$, and
$\tau$ is a weak representation.

\item[(iv)] The mixed commutativity or anti-commutativity identity implies that every element of $\E(X)$ has the form $(f,\varepsilon f)$. The remaining identities with one operator identify $\E(X)$, as a vector space, with $\A_X$ via $j_X$. Using the chosen $\lambda/\mu$-rules, we get
\[
m_X(j_X(f),j_X(g))
=\bigl(m_X^\varepsilon(f,g),
       \varepsilon m_X^\varepsilon(f,g)\bigr).
\]
By linearity, for every $t\in \A_X\otimes \A_X$,
\[
\widetilde m_X\bigl((j_X\otimes j_X)(t)\bigr)
=\bigl(\widetilde m_X^\varepsilon(t),
       \varepsilon\widetilde m_X^\varepsilon(t)\bigr).
\]
This pair belongs to $\E(X)=j_X(\A_X)$ if and only if
$\widetilde m_X^\varepsilon(t)\in \A_X$.
Hence $j_X\otimes j_X$ identifies $D_X^\varepsilon$ with $D_X$,
and $j_X$ preserves the partial multiplications.
\end{itemize}
\end{proof}

\begin{definition}\cite{WRAAlg}
The partial algebra $\E(X)$ is called the \emph{external weak actor} of~$X$, and the pair $(\E(X),\tau)$ is called an \emph{external weak representation} of the functor $\SplExt(-,X)$. When $\tau$ is a natural isomorphism of functors, $\E(X)$ is called the \emph{external actor} of $X$.
\end{definition}

The following result motivates the choice of commutative and anti-commutative varieties considered in the examples below.

\begin{theorem}\cite{WRAAlg}
Let $\bF$ be a field with $\bchar(\bF)\neq 2,3$, and let $\V$ be a non-abelian, action accessible, operadic, quadratic variety of non-associative $\bF$-algebras.

\begin{enumerate}
\item[(1)]
If $\V$ is a variety of commutative algebras, then it coincides
with one of
\[
\CAssoc,\qquad \JJord,\qquad \Nil_2(\Com).
\]

\item[(2)]
If $\V$ is a variety of anti-commutative algebras, then it coincides
with one of
\[
\Lie,\qquad \ACAAssoc,\qquad \Nil_2(\ACom).
\]
\end{enumerate}
\end{theorem}

In the following examples, we describe the external weak actors
of these varieties, as well as those of the abelian variety $\AbAlg$, of the varieties $\Assoc$ and $\Leib$, and of the variety $\Nil_2(\Alg)$ of $2$-nilpotent algebras. We specify the $\lambda/\mu$-rules used in each case.

\begin{examples}\label{ExamplesEX}{\ }
\begin{enumerate}
\item[(1)] If $\V=\AbAlg$, we choose
\[
x(yz)=(yz)x=0.
\]
Then $\E(X)=0$ is the actor of $X$.

\item[(2)]
If $\V=\CAssoc$, we choose
\[
x(yz)=(xz)y,
\qquad
(yz)x=(xz)y.
\]
With this choice, $\E(X)\cong\M(X)$ is an associative
algebra, with multiplication induced by the usual composition of endomorphisms. Moreover, there is a natural isomorphism
\[
\SplExt_{\CAssoc}(-,X)
\cong
\Hom_{\Assoc}(\overline{U}(-),\M(X)),
\]
where $\overline{U}\colon\CAssoc\to\Assoc$ denotes the
forgetful functor and $\M(X)$ is the associative algebra
of multipliers of $X$. Hence $\E(X)\cong\M(X)$ is an external actor of $X$. Since $\M(X)$ need not be commutative, this does not yield a weak representation in $\CAssoc$. In fact, $\CAssoc$ is not weakly action representable; see~\cite{WRAAlgCorr}.

\item[(3)]
If $\V=\Lie$ and $\bchar(\bF)\neq2$, we choose
\[
\begin{aligned}
x(yz)&=(xy)z-(xz)y,\\
(yz)x&=-(xy)z+(xz)y.
\end{aligned}
\]
Then $\E(X)\cong\Der(X)$, with the usual commutator bracket, is the actor of~$X$.

\item[(4)]
Let $\V=\JJord$ be the variety of Jacobi--Jordan algebras. We choose
\[
\begin{aligned}
x(yz)&=-(xy)z-(xz)y,\\
(yz)x&=-(xy)z-(xz)y.
\end{aligned}
\]
For any Jacobi--Jordan algebra $X$, let
\[
\ADer(X)=
\left\{
d\in\End(X)\ \middle|\
d(xy)=-d(x)y-xd(y),\; \forall x,y\in X
\right\}
\]
be the vector space of \emph{anti-derivations} of $X$.
Consider the linear map
\[
\widetilde m_X\colon\ADer(X)\otimes\ADer(X)\to\End(X)
\]
determined on pure tensors by
\[
\widetilde m_X(d\otimes d')=-d\circ d'-d'\circ d,
\]
and put
\[
D_X=\widetilde m_X^{-1}\bigl(\ADer(X)\bigr),\qquad
\mu_X=\widetilde m_X|_{D_X}.
\]
This defines the partial algebra
\[
P_X=\bigl(\ADer(X),D_X,\mu_X\bigr),
\]
whose partial multiplication is commutative.
For $d,d'\in\ADer(X)$, their product is defined if and only if
$-d\circ d'-d'\circ d$ is an anti-derivation.
Whenever it is defined, the product is given by
\[
\langle d,d'\rangle=-d\circ d'-d'\circ d.
\]

To see that $P_X$ is an external actor, let $B$ be an object of $\JJord$. A linear map $\varphi\colon B\to\ADer(X)$ defines a homomorphism
$\widetilde U(B)\to P_X$ if and only if
\[
\varphi(bb')
=-\varphi(b)\circ\varphi(b')
 -\varphi(b')\circ\varphi(b)
\]
for every $b,b' \in B$. The condition $\varphi(b)\in\ADer(X)$ expresses the mixed
Jacobi identities with one element of $B$, while the displayed equation expresses those with two elements of $B$. Thus,
\[
l(b,x)=r(x,b)=\varphi(b)(x)
\]
define a derived action of $B$ on $X$.
Conversely, every derived action has this form and determines a
unique such homomorphism. This correspondence is natural in $B$,
giving
\[
\SplExt_{\JJord}(-,X)
\cong
\Hom_{\PAlg}\bigl(\widetilde U(-),P_X\bigr).
\]
With the above choice of the $\lambda/\mu$-rules,
\Cref{thmV}(iv), applied with $\varepsilon=1$ and
$\alpha=\beta=-1$, gives an isomorphism
$\E(X)\cong P_X$ of partial algebras; see
\cite[Theorem~2.17]{WRAAlg}.
If $\bchar(\bF)=2$, then $\ADer(X)=\Der(X)$ and
$\langle d,d'\rangle=[d,d']$.
In this case, the bracket is symmetric
and satisfies the Jacobi identity.
Hence, $P_X$ is an object of $\JJord$, and the above natural isomorphism shows that $\Der(X)$ is the actor of $X$. Thus, $\JJord$ is action representable in
characteristic $2$; this also follows from~\cite[Theorem~1.12]{Tim}, since in characteristic~$2$ the variety $\JJord$ coincides with the variety of quasi-Lie algebras considered there. When $\bchar(\bF)\neq2$, it remains open whether $\JJord$ is weakly action representable.

\item[(5)]
Let $\V=\ACAAssoc$ be the variety of anti-commutative
anti-associative algebras. We choose
\[
x(yz)=(xz)y,
\qquad
(yz)x=-(xz)y.
\]
For any anti-commutative anti-associative algebra $X$, we define the vector space
\[
\A_X=
\left\{
f\in\End(X)\ \middle|\
f(xy)=-f(x)y,\; \forall x,y\in X
\right\}.
\]
Consider the linear map
\[
\widetilde m_X\colon \A_X\otimes \A_X\to\End(X) \colon f\otimes g \mapsto -f\circ g,
\]
and put
\[
D_X=\widetilde m_X^{-1}(\A_X),\qquad
\mu_X=\widetilde m_X|_{D_X}.
\]
This defines the partial algebra
\[
P_X=(\A_X,D_X,\mu_X).
\]
For $f,g\in \A_X$, their product is defined if and only if
$-f\circ g\in \A_X$, and then
\[
\langle f,g\rangle=-f\circ g.
\]
The partial multiplication is associative wherever both iterated
products are defined, since
\[
\langle\langle f,g\rangle,h\rangle
=(f\circ g)\circ h=f \circ (g \circ h)
=\langle f,\langle g,h\rangle\rangle.
\]
Let $B$ be an object of $\ACAAssoc$. A linear map $\varphi\colon B\to \A_X$
defines a homomorphism $\widetilde U(B)\to P_X$ if and only if
\[
\varphi(bb')=-\varphi(b)\circ\varphi(b')
\]
for every $b,b' \in B$. The condition $\varphi(b)\in \A_X$ expresses the mixed
anti-associative identities with one element of $B$.
The displayed equation, together with the anti-commutativity
of $B$, expresses those with two elements of $B$.
Thus
\[
l(b,x)=-r(x,b)=\varphi(b)(x)
\]
define a derived action of $B$ on $X$. Conversely, every derived
action determines a unique such homomorphism.
Naturality in $B$ therefore gives
\[
\SplExt_{\ACAAssoc}(-,X)
\cong
\Hom_{\PAlg}\bigl(\widetilde U(-),P_X\bigr).
\]
With the above choice of the $\lambda/\mu$-rules,
\Cref{thmV}(iv), applied with $\varepsilon=-1$,
$\alpha=0$ and $\beta=1$, yields an isomorphism
$\E(X)\cong P_X$ of partial algebras; see
\cite[Theorem~2.20]{WRAAlg}.
It remains open whether $\ACAAssoc$ is weakly action representable.

\item[(6)]
If $\V=\Nil_2(\Com)$ or $\V=\Nil_2(\ACom)$, we choose
\[
x(yz)=(yz)x=0.
\]
Then $\E(X)$ is isomorphic to the abelian algebra
\[
\A_X=\left\{
f\in\End(X)
\mid
f(xy)=f(x)y=0,\; \forall x,y\in X
\right\}.
\]
It follows from \Cref{thmV}(iii) that both
$\Nil_2(\Com)$ and $\Nil_2(\ACom)$ are weakly action
representable; see also~\cite[Theorems~2.19 and~2.21]{WRAAlg}.

\item[(7)] If $\V=\Assoc$ and we choose the $\lambda/\mu$-rules as in \Cref{Example associative}, then $\E(X)\cong\Bim(X)$, the associative algebra of bimultipliers of $X$. Thus, as already stated in \Cref{USGA_Assoc}, the variety $\Assoc$ is weakly action representable.

\item[(8)] Let $\V=\Leib$ and choose the $\lambda/\mu$-rules as in \Cref{esLeib}. Then $\E(X)\cong\Bider(X)$, the Leibniz algebra of biderivations of $X$. It follows that the variety $\Leib$ is weakly action representable; see~\cite{CigoliManciniMetere}.

\item[(9)]
If $\V=\Nil_2(\Alg)$, we choose again
\[
x(yz)=(yz)x=0.
\]
Then, as observed in \Cref{exnil2}, $\E(X)$ is an abelian
algebra. Thus, $\Nil_2(\Alg)$ is weakly action representable.
\end{enumerate}
\end{examples}

\begin{remark}
By the results recalled at the end of \Cref{SecIntAct}, there are action
accessible varieties~$\V$ for which no choice of the $\lambda/\mu$-rules
makes $\E(X)$ an algebra of~$\V$ for every object $X$ of~$\V$.
Indeed, the varieties $\Nil_k(\Lie)$ of $k$-nilpotent and $\Sol_n(\Lie)$ of $n$-solvable Lie algebras are not weakly action representable for $k\geq3$ and $n\geq2$, respectively~\cite{XabiMancini}, and neither is the variety $\CAssoc$~\cite{WRAAlgCorr}.
These are action accessible, operadic varieties (for $\Nil_k(\Lie)$ and $\Sol_n(\Lie)$, over a field of characteristic different from~$2$), so that the assertion follows from \Cref{thmV}(iii).
\end{remark}

\begin{remark}
The construction of the vector space $\E(X)$ can also be carried out for an operadic variety $\V$ which is not action accessible. In that case, however, there is no canonical way to
endow $\E(X)$ with a bilinear map $m_X$ as in
\Cref{E(X)}. We therefore obtain only a natural monomorphism of
functors
\[
\tau\colon
\SplExt(-,X)
\rightarrowtail
\Hom_{\Vect}(\hat{U}(-),\E(X)),
\]
where $\Vect$ is the category of $\bF$-vector spaces and
$\hat{U}\colon\V\to\Vect$ is the forgetful functor.
\end{remark}

\subsection{The USGA and the external weak actor}\label{SecRel}

As described in \Cref{SecUSGA}, for every object $X$ of an Orzech category of interest $\cC$, there is a natural monomorphism of functors
\[
\mu\colon
\SplExt(-,X)
\rightarrowtail
\Hom_{\cC'}(U(-),\USGA(X)),
\]
where $\cC'$ is a category containing $\cC$ as a full subcategory,
$U\colon\cC\to\cC'$ denotes the forgetful functor, and
$\USGA(X)$ is an object of $\cC'$ called the \emph{universal strict
general actor} of $X$. We further recall that, once a presentation of
the Orzech category of interest $\cC$ has been fixed, $\USGA(X)$ is
unique up to isomorphism; see \Cref{presUSGA}.

\begin{remark}
\Cref{Theorem AC iff Orzech} shows that a homogeneous variety of
non-associative algebras is an Orzech category of interest if and
only if it is action accessible. As observed in \Cref{USGA_var},
a formulation of Axiom~2 is determined by a choice of constants
\[
\lambda_1,\ldots,\lambda_8,\mu_1,\ldots,\mu_8\in\bF
\]
defining the $\lambda/\mu$-rules. Moreover, in this case, $\V'=\Alg$, the category of all non-associative algebras, since we regard the opposite operation as a derived operation; see \Cref{SecOrz}.
\end{remark}

The results recalled in \Cref{USGA} and \Cref{thmV} give the following natural monomorphisms for varieties of non-associative algebras.

\begin{proposition}
Let $\V$ be an action accessible, operadic variety of
non-associative algebras, together with a choice of constants
\[
\lambda_1,\ldots,\lambda_8,\mu_1,\ldots,\mu_8\in\bF
\]
defining the $\lambda/\mu$-rules. For every object $X$ of $\V$,
there are natural monomorphisms of functors
\[
\mu\colon
\SplExt(-,X)
\rightarrowtail
\Hom_{\Alg}(U(-),\USGA(X))
\]
and
\[
\tau\colon
\SplExt(-,X)
\rightarrowtail
\Hom_{\PAlg}(\widetilde U(-),\E(X)),
\]
where
\[
U\colon\V\to\Alg
\qquad\text{and}\qquad
\widetilde U\colon\V\to\PAlg
\]
are the forgetful functors.
\end{proposition}

\begin{remark}\label{exUSGAvsE}
The universal strict general actor and the external weak actor are different in general. Indeed, by construction, $\USGA(X)$ is generated by the elements which arise from derived actions on~$X$ (see~\cite[Section~4]{Casas}), whereas $\E(X)$ may contain further elements, even when its multiplication is total. For instance, let $\V=\Nil_2(\Alg)$, with the $\lambda/\mu$-rules of \Cref{exnil2}, and let $X=\bF$ be the one-dimensional abelian algebra. If $B$ is a $2$-nilpotent algebra acting on~$X$, then, for every $b\in B$, the identities $x(yz)=0$ and $(yz)x=0$ in $B\ltimes X$ give $b\ast(b\ast x)=0$ and $(x\ast b)\ast b=0$ for all $x\in X$. Since $b\ast-$ and $-\ast b$ are scalar multiplications on $X=\bF$, it follows that $b\ast x=0=x\ast b$ for all $x\in X$. Hence every derived action on $X$ is trivial, so that $\USGA(X)=0$, which is in fact the actor of~$X$. On the other hand, since $X$ is abelian, all the conditions defining $\E(X)$ in \Cref{exnil2} are trivially satisfied, so that $\E(X)=\End(X)^2\cong\bF^2$, with zero multiplication.

However, it is often more convenient to work with the external weak
actor~$\E(X)$, since its construction is simpler than that of the
universal strict general actor $\USGA(X)$.
\end{remark}

\section{Conclusions and open questions}\label{SecConcl}

In this survey, we studied the representability of actions in
varieties of non-associative algebras over a field. After recalling
the general categorical framework, we examined action
representability, weak action representability, and action
accessibility, together with the constructions of the universal
strict general actor and the external weak actor. We first summarise the known results for the varieties considered above, and then outline some possible directions for future research.

\subsection{Summary of known results}

\Cref{tab:summary} collects the known results on the representability of actions in the varieties of non-associative algebras considered in this survey, over a field of characteristic different from~$2$, together with the further hypotheses on the field and the cases that remain open. For the weakly action representable varieties, it also indicates a weak representing object $T(X)$ of an object~$X$.

\begin{table}[ht]
\caption{Representability of actions in the varieties considered in this survey, over a field $\bF$ with $\bchar(\bF)\neq2$. Here AR, WAR and AA stand for action representability, weak action representability and action accessibility, respectively, so that AR $\Rightarrow$ WAR $\Rightarrow$ AA. When AR holds, $T(X)=[X]$ is the actor of~$X$.}
\label{tab:summary}
\centering
\small
\begin{tabular}{@{}lccccl@{}}
\toprule
Variety & AR & WAR & AA & $T(X)$ & References\\
\midrule
$\AbAlg$ & yes & yes & yes & $0$ & \Cref{ExamplesEX}(1)\\
$\Lie$ & yes & yes & yes & $\Der(X)$ & \Cref{exliealgebras}\\
$\Assoc$ & no$^{b}$ & yes & yes & $\Bim(X)$ & \Cref{thmassoc}\\
$\Leib$ & no$^{a}$ & yes & yes & $\Bider(X)$ & \Cref{ExamplesEX}(8)\\
$\Nil_2(\Com)$ & no$^{a}$ & yes & yes & $\E(X)$ & \Cref{ExamplesEX}(6)\\
$\Nil_2(\ACom)$ & no$^{a}$ & yes & yes & $\E(X)$ & \Cref{ExamplesEX}(6)\\
$\Nil_2(\Alg)$ & no$^{a}$ & yes & yes & $\E(X)$ & \Cref{ExamplesEX}(9)\\
$\JJord$ & no$^{a}$ & open & yes & --- & \Cref{ExamplesEX}(4)\\
$\ACAAssoc$ & no$^{a}$ & open & yes & --- & \Cref{ExamplesEX}(5)\\
$\Nil_k(\Assoc)$, $k\geq3$ & no$^{a}$ & open & yes & --- & \Cref{Example Nilk}\\
$\Alt$ & no$^{a}$ & open & yes & --- & \Cref{exalt}\\
$\CAssoc$ & no$^{c}$ & no & yes & --- & \Cref{ExamplesEX}(2), \cite{WRAAlgCorr}\\
$\Nil_k(\Lie)$, $k\geq3$ & no & no & yes & --- & \cite{XabiMancini}\\
$\Sol_n(\Lie)$, $n\geq2$ & no & no & yes & --- & \cite{XabiMancini}\\
\bottomrule
\end{tabular}

\bigskip
\begin{minipage}{\linewidth}
\footnotesize
$^{a}$~Over an infinite field, by \Cref{AR implies Lie}.\\
$^{b}$~See~\cite[Proposition~1.11]{Tim}.\\
$^{c}$~See~\cite[Theorem~2.6]{BJK2} and \Cref{SecAR}.
\end{minipage}
\end{table}

Beyond semi-abelian categories, the external weak actor was used in~\cite{ManciniPiazzaVienne} to study actions of unital algebras in the context of \emph{ideally exact categories}~\cite{IdeallyExact}: for an action accessible, operadic, unit-closed variety $\V$, the category of unital algebras of~$\V$ is action representable, and the actor of a unital algebra $X$ is $X$ itself.

\subsection{Categorical and algebraic properties of $\PAlg$}

As observed in \Cref{remzeroPAlg}, the category $\PAlg$ of partial algebras over a field is pointed, and its zero object is~$\widetilde U(0)$, the zero algebra viewed as a partial algebra. A natural direction for future work is to investigate its categorical and algebraic properties, with particular attention to the description of internal actions in $\PAlg$. In this context, the phenomenon described in \Cref{rem:pure} has to be taken into account.

\subsection{A converse to \Cref{thmV}(iii)}

The criterion of \Cref{thmV}(iii) is sufficient, but not necessary, for the weak action representability of a variety.
For instance, $\Leib$ is weakly action representable, but, with the second choice of the $\lambda/\mu$-rules in \Cref{exLeib}, the external weak actor of a Leibniz algebra need not be a Leibniz algebra. It is therefore natural to seek additional hypotheses under which the failure of $\E(X)$ to
satisfy the defining identities of $\V$, for some object~$X$, implies
that $\V$ is not weakly action representable. More generally, one may
ask whether the identities satisfied by $\E(X)$ provide effective
obstructions to weak action representability.

Over a field of characteristic different from $2$, the varieties
$\Nil_k(\Assoc)$, for $k\geq3$, and $\Alt$, considered in
\Cref{Example Nilk,exalt}, as well as $\JJord$ and $\ACAAssoc$,
provide relevant test cases.
With the choices made above, their external weak actors
need not belong to the original variety.

For instance, let $X\neq0$ be an abelian algebra.
In $\ACAAssoc$, its external weak actor is $\End(X)$ with
multiplication
\[
\langle f,g\rangle=-f\circ g.
\]
This multiplication is not anti-commutative, since
\[
\langle\id_X,\id_X\rangle=-\id_X\neq0.
\]
In $\JJord$, the external weak actor is $\End(X)$ with
multiplication
\[
\langle f,g\rangle=-f\circ g-g\circ f.
\]
For all $f,g,h\in\End(X)$, one computes
\[
\langle f,\langle g,h\rangle\rangle+\langle g,\langle h,f\rangle\rangle+\langle h,\langle f,g\rangle\rangle
=2\,(fgh+fhg+gfh+ghf+hfg+hgf),
\]
where juxtaposition denotes composition. If $\bchar(\bF)\neq 2,3$, the Jacobi identity evaluated at
$(\id_X,\id_X,\id_X)$ gives
\[
3\,\langle\id_X,\langle\id_X,\id_X\rangle\rangle
=12\,\id_X\neq0.
\]
If $\bchar(\bF)=3$ and $\dim_{\bF}X\geq2$, choose linearly independent $e_1,e_2\in X$ and a complement $W$ of their span in $X$, and let $E_{ij}\in\End(X)$, for $i,j\in\{1,2\}$, be defined by $E_{ij}(e_l)=\delta_{jl}e_i$ for $l\in\{1,2\}$ and $E_{ij}(W)=0$. Then the Jacobi identity evaluated at $(E_{11},E_{12},E_{21})$ gives $2(2E_{11}+E_{22})\neq0$.
Hence, in both cases, $\E(X)$ is not an object of $\JJord$. On the other hand, if $\bchar(\bF)=3$ and $X=\bF$, then $\E(X)\cong\bF$ with multiplication $\langle a,b\rangle=-2ab=ab$, which is a Jacobi--Jordan algebra, since $3abc=0$.

\subsection{External weak actors for varieties of algebras with several bilinear operations}

The construction of the external weak actor~$\E(X)$ introduced in \Cref{SecRepVar} has been extended in~\cite{ManciniSeveral} to arbitrary action accessible, operadic varieties of algebras with $n\geq1$ bilinear operations, none of which is assumed to be associative. Further open problems in this setting are discussed in the same article.

\section*{Acknowledgements}
The author is grateful to his coauthors José Brox, Alan S.~Cigoli, Xabier García-Martínez, Giuseppe Metere, Tim Van der Linden and Corentin Vienne, with whom the results presented in this survey were obtained. He would like to thank in particular Giuseppe Metere, who supervised his doctoral thesis~\cite{Mancini_Thesis}, on which this survey is based, and Tim Van der Linden, for his valuable comments and for the discussions during the writing of this survey. He also thanks Andrea Albano for his careful reading of the manuscript.

The author is supported by the ``National Group for Algebraic and Geometric Structures and their Applications'' (GNSAGA -- INdAM), and he is a Postdoctoral Researcher of the Fonds de la Recherche Scientifique--FNRS.


\section*{Use of generative AI tools}
Generative AI tools (Anthropic's Claude Opus~5.5 and OpenAI's GPT-6~Sol) were used to polish the exposition and to help check the correctness of the proofs; they were not used to obtain any mathematical results. The author reviewed all suggestions made by these tools and takes full responsibility for the content of this article.

\vspace{1.5\baselineskip}

\end{document}